\documentclass[11pt,a4paper]{amsart}
\usepackage[a4paper,centering,left=3.4cm,right=3.4cm]{geometry}
\usepackage{amsfonts,amscd,amssymb,amsmath,amsthm,mathrsfs,multirow,xcolor,lscape}
\usepackage{longtable,pbox,lipsum, stmaryrd,url, epstopdf, float, enumerate}
\usepackage[all,cmtip]{xy}
\usepackage[english]{babel}
\usepackage{color}

\usepackage{ucs}
\usepackage[utf8]{inputenc}
\usepackage{dsfont}
\usepackage[T1]{fontenc} 
\usepackage{lmodern}
\usepackage[thinlines]{easytable}

\usepackage{phaistos}
\usepackage{microtype}

\usepackage{enumitem}
\usepackage{scalerel}
\usepackage{todonotes}
\usepackage{stackengine,wasysym}
\usepackage{todonotes}
\usepackage[T1]{fontenc}

\usepackage[colorlinks,linkcolor=blue,anchorcolor=blue,citecolor=blue,backref=page]{hyperref}
\usepackage{hyperref}

\newtheorem{thm}{Theorem}[section]

\newtheorem{defi}[thm]{Definition}
\newtheorem{lemma}[thm]{Lemma}

\newtheorem{coro}[thm]{Corollary}
\newtheorem{rmk}[thm]{Remark}

\newtheorem{conjecture*}{Conjecture}
\newtheorem{notation}[thm]{Notation}

\newtheorem{pro}[thm]{Property}
\newtheorem*{theorem*}{Main Theorem}

\usepackage[all,cmtip]{xy}
\usepackage{tikz}
\usetikzlibrary{arrows} 
\usetikzlibrary{matrix}
\usetikzlibrary{calc}

\newcommand{\rB}{\mathrm{B}}
\newcommand{\rG}{\mathrm{G}}

\newcommand{\rN}{\mathrm{N}}

\newcommand{\rT}{\mathrm{T}}

\newcommand{\BA}{{\mathbb{A}}}

\newcommand{\BC}{{\mathbb{C}}}

\newcommand{\BI}{{\mathbb{I}}}

\newcommand{\BQ}{{\mathbb{Q}}}
\newcommand{\BR}{{\mathbb{R}}}

\newcommand{\BZ}{{\mathbb{Z}}}

\newcommand{\Hom}{\mathop{\rm Hom}\nolimits}

\newcommand{\Sym}{\mathop{\rm Sym}\nolimits}

\newcommand{\SL}{\mathop{\rm SL} \nolimits}

\newcommand{\GL}{\mathop{\rm GL} \nolimits}

\newcommand{\U}{\mathop{\rm U} \nolimits}
\newcommand{\SU}{\mathop{\rm SU} \nolimits}

\newcommand{\dm}{\mathop{\rm d} \nolimits}

\newcommand{\Eis}{\mathop{\rm Eis} \nolimits}

\newcommand{\iu}{\mathsf{i}}

\newcommand{\tmt}[4]{\left({#1\atop #3}{#2\atop #4}\right)}

\newcommand{\q}[1]{``#1''}
\newcommand\restr[2]{{
  \left.\kern-\nulldelimiterspace 
  #1 
  \right|_{#2} 
  }}

\title[Relative Lie algebra cohomology and Eisenstein classes of $\SU(2,1)$]{Relative Lie algebra cohomology of $\SU(2,1)$ and Eisenstein classes on Picard surfaces}  
\author{Jitendra Bajpai} 
\address{Department of Mathematics, University of Kiel, Germany.}
\email{jitendra@math.uni-kiel.de}
\author{Mattia Cavicchi}
\address{Laboratoire de Mathématiques d'Orsay, Orsay, France}
\email{mattia.cavicchi@universite-paris-saclay.fr}

\begin{document}
\date{\today}
\begin{abstract}
We consider Picard surfaces, locally symmetric varieties $S_{\Gamma}$ attached to the Lie group $\SU(2,1)$, and we construct explicit differential forms on $S_{\Gamma}$ representing Eisenstein classes, i.e. cohomology classes restricting non-trivially to the boundary of the Borel-Serre compactification. This is needed for the computation of the class of the extensions of the Hodge structure that we have constructed in \cite{BC22} according to the predictions of the Bloch-Beilinson conjectures. The tool for the construction of the differential forms is an analysis of relative Lie algebra cohomology of the principal series of $\SU(2,1)$ using recent methods of Buttcane and Miller. 
\end{abstract}

\maketitle

\tableofcontents


\section{Introduction}\label{intro}

This work is part of the program initiated by the two authors in~\cite{BC22}, aimed at studying the Bloch-Beilinson conjectures for Hecke characters through the cohomology of Picard surfaces. The purpose of this article, whose results are completely independent from the ones in~\cite{BC22}, is in fact to construct explicit Eisenstein differential forms on such surfaces. To specify what we mean by \q{explicit} and \q{Eisenstein}, to explain our methods, and to describe the interest of the results for our program, let us give some context. 

It is a general principle in the Langlands program (\cite{BG14}) that automorphic forms, which are \emph{algebraic} in an appropriate sense, should be attached to Galois representations, in such a way that $L$-functions on the automorphic and Galois side are matched. In this case, the zeroes of the $L$-function at suitable integers should carry arithmetic information of considerable interest. The classical setting, which has been the motivation for this whole circle of ideas since the '50s, is provided by algebraic Hecke characters, i.e. algebraic automorphic forms for $\GL_1$. To these, one knows how to attach, since Weil, a compatible system of $\ell$-adic Galois representations, which arise from geometry - more precisely, from a \emph{Chow motive} (\cite{DM91}). 

In the case of an algebraic Hecke character of odd weight $w$, the central point $\frac{w+1}{2}$ of the functional equation is an integer. The Bloch-Beilinson conjectures (\cite{Nek94}) predict then what to expect when $L(\phi,s)$ vanishes at $s=\frac{w+1}{2}$. There should exist a non-trivial extension of motives of a specific form, with an attached period and height pairing, in terms of which one should be able to express the leading term of the $L$-function at $s=\frac{w+1}{2}$. 

We have started investigating the case of algebraic Hecke characters of a quadratic imaginary number field in \cite{BC22}. For Hecke characters $\phi$ of weight $-3$ and of a special shape, whose $L$-function vanishes at the central point $s=-1$, we have launched an attempt to construct the expected \emph{extensions of mixed Hodge structure}. Following the insights of Harder \cite{Har93}, the source for the construction is the geometry of Picard surfaces $S_{\Gamma}$. The latter are quasi-projective complex algebraic surfaces, arising as ball quotients by a congruence subgroup $\Gamma$ of a non-split unitary group $\rG$, with real points $G \simeq \SU(2,1)$. The main result of \cite{BC22} uses the cohomology of $S_{\Gamma}$ to show the existence of extensions, which have the desired form, up to the computation of their extension class, which in particular, is not yet known to be non-zero. 

Let us explain the principle of construction. The non-compact $S_{\Gamma}$ admits a natural \emph{Borel-Serre compactification} $\overline{S_{\Gamma}}$, such that the open immersion 
$S_{\Gamma} \hookrightarrow \overline{S_{\Gamma}}$ is an homotopy equivalence. The singular cohomology spaces $H^{\bullet}(S_{\Gamma})$  and $H^{\bullet}(\overline{S_{\Gamma}})$ are then isomorphic and this produces a restriction map $r$ from $H^{\bullet}(S_{\Gamma})$ to the cohomology of the boundary $\partial \overline{S_{\Gamma}}$. Its image is called \emph{Eisenstein cohomology} $H^{\bullet}_{\Eis}(S_{\Gamma})$ and classes in $H^{\bullet}(S_{\Gamma})$ with non-zero image under $r$ are called \emph{Eisenstein classes}. The existence of suitable such classes in degree 2, depending on the behaviour of the $L$-function of the Hecke character that we are studying, is one of the deep inputs used in~\cite{BC22}. It allows us to exploit the exact sequence
\[
0 \rightarrow H^2_!(S_{\Gamma}) \rightarrow H^2(S_{\Gamma}) \rightarrow H^2_{\Eis}(S_{\Gamma}) \rightarrow 0 
\]
(where $H^2_!(S_{\Gamma})$ is the image in $H^2(S_{\Gamma})$ of cohomology with compact supports) to start constructing the looked-for extensions. The first goal of our program is to show their non-triviality. 

For this, we need to control the properties of the Eisenstein classes that we employ: namely, to exhibit explicit differential forms on $S_{\Gamma}$ representing them, such that we understand their position in the \emph{Hodge filtration}. This means that we need to understand their \emph{Hodge types} and their behaviour near the boundary of a \emph{toroidal compactification} of $S_{\Gamma}$. This article aims to provide representatives amenable to such an analysis. In a future work~\cite{BC}, we will carry out the study of their Hodge theory, and use that information to prove the non-triviality of the extensions constructed in~\cite{BC22}. 

In~\eqref{character_k}, we define a family of characters $\phi_{\infty}$, depending on an integer $k \geq 0$, of a maximal torus $\rT$ of the group $\rG$ underlying $S_{\Gamma}$. The choice of these characters is imposed by the fact that they give the components at infinity of the Hecke characters that we want to consider. The cohomology $H^2_{\Eis}(S_{\Gamma})$ is then expressed in terms of relative Lie algebra cohomology $H^2(\mathfrak{g}, \mathfrak{l}, I_{\phi_{\infty}} \otimes V_k)$ of the principal series representation $I_{\phi_{\infty}}$ of $G$ induced by $\phi_{\infty}$, tensored with the $k$-symmetric power $V_k$ of the standard representation $V$ of $G$, where $\mathfrak{g}$ is the Lie algebra of $G$ and  $\mathfrak{l}$ is the Lie algebra of the maximal compact subgroup $K_{\infty}$ of $G$. This cohomology space is known to be 1-dimensional (cfr.~Corollary~\ref{1dim}). Our main result, leading to the construction of explicit differential forms on $S_{\Gamma}$, is then the following. For the sake of this statement, we use the terminology \emph{differential forms at the boundary} for elements of the Chevalley-Eilenberg complex computing $H^{\bullet}(\mathfrak{g}, \mathfrak{l}, I_{\phi_{\infty}} \otimes V_k)$.

\begin{theorem*}{(Theorem~\ref{mainthm})}
Fix an integer $k \geq 0$. Any generator of $H^2(\mathfrak{g}, \mathfrak{l}, I_{\phi_{\infty}} \otimes V_k)$ is represented by a linear combination of explicit differential forms $\psi^k$ and $\psi^k_0$ at the boundary. The forms $\psi^k$ and $\psi^k_0$ are characterized, up to a non-zero constant, by the property of having type respectively $(1,1)$ and $(0,2)$. 
\end{theorem*}

In the above statement, the term \emph{type} is defined in terms of the bigrading on the complexified tangent space of $S_{\Gamma}$ (see~\eqref{bigrading}), arising from the complex structure carried by the locally symmetric space. It is related to the natural Hodge structure on $H^2(S_{\Gamma})$ - whose analysis, however, we do not pursue here. Indeed, we limit ourselves to explain in Corollary \ref{finalcoro} how the main theorem allows us to construct differential forms representing Eisenstein classes on $S_{\Gamma}$, but we will study their position in the Hodge filtration in the forthcoming paper~\cite{BC}. 

Let us comment on the relation of our main theorem with the existing literature and the ingredients of its proof. The relative Lie algebra cohomology of principal series representations has already been studied and computed in great generality in the foundational work~\cite{BoWa}. However, these methods do not provide us with explicit elements of the Chevalley-Eilenberg complex representing the classes we are interested in. We need to study their types (in the sense explained above) and the elements of $I_{\phi_{\infty}}$ appearing in their expression, as a crucial step towards the understanding of the associated Eisenstein classes at the boundary. Hence, we need first a concrete hold on the principal series representation $I_{\phi_{\infty}}$ itself, as a $(\mathfrak{g}, K_{\infty})$-module. 

In principle, the $K_{\infty}$-finite vectors in $I_{\phi_{\infty}}$ have been described in the classical paper~\cite{JW77}, by means of spherical harmonics on $\mathbb{S}^3$. The latter is diffeomorphic to $\SU(2)$ and hence to the derived subgroup of $K_{\infty}$. This description is however unwieldy for our needs. On the other hand, Buttcane and Miller have recently started in~\cite{BM19} an analysis of $K_{\infty}$-finite vectors in principal series representations, with explicit formulae for the $\mathfrak{g}$-action, by a method which can be applied to Lie groups $G$ with maximal compact subgroup $K_{\infty}$ isogenous to a product of copies of $\U(1)$ and $\SU(2)$. The case of interest for us, $G=\SU(2,1)$, has been carried out in detail by Z. Zhang in the preprint~\cite{Zha19}. The method resorts to the explicit parametrization of $\U(2)$ by \emph{Euler angles} and to the expression of $K_{\infty}$-finite vectors through \emph{Wigner $D$-functions} on $\U(2)$, i.e. matrix coefficients of the finite dimensional representations of the latter group. The technical heart of our work (Section~\ref{relativeLie}) consists then in dealing with the explicit formulas resulting from this analysis, and in using them to construct suitably explicit differential forms at the boundary of $S_{\Gamma}$. 

\subsection{Outline} In Section~\ref{structure}, we recall the basic information on the structure of the group $\SU(2,1)$ and of its maximal compact subgroups $K_{\infty}$, along with their Lie algebras $\mathfrak{g}$ and $\mathfrak{l}$ respectively. In Section \ref{princ}, we use this material to define the principal series representations $I_{\phi_{\infty}}$ of $G$, induced by the characters $\phi_{\infty}$ of the shape we are interested in, and we recall the results of \cite{Zha19} on their $K_{\infty}$-vectors and their structure as $(\mathfrak{g}, K_{\infty})$-modules. In Section \ref{relativeLie}, we build on these results to prove our main theorem (Theorem \ref{mainthm}), constructing explicit representatives of a generator of $H^2(\mathfrak{g}, \mathfrak{l}, I_{\phi_{\infty}} \otimes V_k)$. Finally, in Section~\ref{eisclasses}, we recall some results of Harder~\cite{Har87b} to explain how Theorem~\ref{mainthm} leads to the construction of differential forms representing Eisenstein classes on Picard surfaces (Corollary~\ref{finalcoro}). 

\subsection{Notations}
An empty entry in a matrix will mean that the given entry is equal to 0. For a matrix $A$, we will denote by $A^{\top}$ its transpose. If $A$ is a $n$-by-$n$ matrix and $\gamma \in \GL_n$, we will denote by $A^{\gamma}$ the conjugate $\gamma A \gamma^{-1}$, and by $H^{\gamma}$ the group $\gamma H \gamma^{-1}$ whenever $H$ is a subgroup of $\GL_n$.  

We will fix once and for all $\iu \in \BC$ such that $\iu^2=-1$. We denote by $z \mapsto \overline{z}$ complex conjugation on $\BC$ and by $v \mapsto \overline{v}$, $v \in \BC^n$, the induced conjugation on $\BC^n$ for any positive integer $n$. If $W$ is any $n$-dimensional $\BC$-vector space and $w \in W$, the notation $\overline{w}$ will stand for the complex conjugate of $w$ under an isomorphism $W \simeq \BC^n$ (which will be clear from the context).  

\section{Structure of $\SU(2,1)$ and $\U(2)$ and of their Lie algebras}\label{structure}

\subsection{The group $G$ and its maximal compact subgroups}\label{basics}
Fix a 3-dimensional complex vector space $V$, equipped with a hermitian form $J$ of signature $(2,1)$. 

\begin{defi}\label{group} The real Lie group $G$ is the group $\SU(2,1)$ of linear automorphisms of $V$ of determinant 1 preserving the form $J$.
\end{defi}
\begin{defi}\label{bases}$\,$
\begin{enumerate}
\item A \emph{diagonal basis} of $V$ is a $\BC$-basis of $V$ in which $J$ acquires the form 
\begin{equation*} 
\left(
\begin{array}{ccc}
1 &   & \\
  & 1 &  \\
  &    & -1 
\end{array}
\right)
\end{equation*}
\item A \emph{parabolic basis} of $V$ is a $\BC$-basis of $V$ in which $J$ acquires the form 
\begin{equation*} 
\left(
\begin{array}{ccc}
 &   & 1 \\
  & 1 &  \\
 1 &    &  
\end{array}
\right)
\end{equation*}
\end{enumerate}
\end{defi}
The two matrices representing $J$ respectively in a diagonal and in a parabolic basis are congruent to each other via 
\[
\gamma:=\left(
\begin{array}{ccc}
\frac{1}{\sqrt2} & & \frac{1}{\sqrt2} \\
& 1 & \\
\frac{1}{\sqrt2} & & -\frac{1}{\sqrt2}
\end{array}
\right)
\]
Whenever the choice of a basis is understood from the context, we will make no distinction in notation between the form $J$ and a matrix representing it. 

Any maximal compact subgroup of $G$ is isomorphic to the real Lie group $\mathrm{U}(2)$. We fix the maximal compact subgroup $K_{\infty}$, given in a diagonal basis by 
\begin{equation*} \label{maxcomp}
K_{\infty}:=
\left\{
\left(
\begin{array}{cc}
U & \\
 & \det(U)^{-1}
\end{array}
\right)
\ \vert \ U \in \mathrm{U} (2)
\right\}
\end{equation*}
In order to provide an explicit parametrization of $K_{\infty}$, observe that $\mathrm{U}(2) \simeq \mathrm{U}(1) \ltimes \mathrm{SU}(2)$
where the compact real Lie group $\mathrm{SU}(2)$ is defined by
\[
	\mathrm{SU}(2) = \left\{\tmt{\alpha}{-\bar{\beta}}{\beta}{\bar{\alpha}} | \alpha, \beta\in\mathbb{C}\text{ and }|\alpha|^2+|\beta|^2=1\right\}
\]
 
It is diffeomorphic to the unit 3-sphere $\mathbb{S}^3$, and we will parametrize it by \emph{Euler angles} as the group of matrices (cfr.~\cite[(2.40)]{BL81}) 
\begin{equation*}
\mathrm{SU}(2)=
\left\{ \left(\begin{smallmatrix}
	e^{-\frac{\mathsf{i}}{2}(\phi+\psi)}\cos\frac{\theta}{2}&-e^{\frac{\mathsf{i}}{2}(\phi-\psi)}\sin\frac{\theta}{2}\\
	e^{\frac{\mathsf{i}}{2}(-\phi+\psi)}\sin\frac{\theta}{2}&e^{\frac{\mathsf{i}}{2}(\phi+\psi)}\cos\frac{\theta}{2}
	\end{smallmatrix}\right) \vert \phi \in(-\pi,\pi],\theta\in[0,\pi],\psi\in(-\pi,3\pi] \right\}
\end{equation*}

By parametrizing
\[
U(1) \simeq \left\{ \left(\begin{smallmatrix} e^{-\frac{\mathsf{i} \zeta}{2}} & \\
& e^{-\frac{\mathsf{i} \zeta}{2}}
\end{smallmatrix}\right) \vert \zeta \in \BR \right\}
\]
as a subgroup of $\mathrm{U}(2)$, we get
{\scriptsize 
\begin{align} \label{param_K}
\begin{split}
\mathrm{U}(2) =\left\{ \left(\begin{smallmatrix}
	e^{\frac{\mathsf{i}}{2}(-\zeta-\phi-\psi)}\cos\frac{\theta}{2}&-e^{\frac{\mathsf{i}}{2}(-\zeta+\phi-\psi)}\sin\frac{\theta}{2}\\
	e^{\frac{\mathsf{i}}{2}(-\zeta-\phi+\psi)}\sin\frac{\theta}{2}&e^{\frac{\mathsf{i}}{2}(-\zeta+\phi+\psi)}\cos\frac{\theta}{2}
	\end{smallmatrix}\right)| \zeta \in \BR, \phi \in(-\pi,\pi], \theta\in[0,\pi], \psi\in(-\pi,3\pi] \right\}
\end{split}
\end{align}
}
Let us now describe the parabolic subgroups of $G$. To do so, we switch to parabolic bases. Fix such a basis, consider the corresponding embedding $G \hookrightarrow \SL_3(\BC)$ and denote by $\Delta$ the subgroup of upper triangular matrices of $\SL_3(\BC)$. We define the standard Borel $B$ of $G$ to be 
\begin{equation} \label{borel}
B : = G \cap \Delta
\end{equation}
The group $B$ is then isomorphic to the semidirect product $T \ltimes N$
where the maximal torus $T$ of $G$ is given by
\begin{equation} \label{maxtorus}
T:=\left\{
\left(
\begin{array}{ccc}
re^{\mathsf{i} t} & & \\
 & e^{-2 \mathsf{i} t} & \\
 & & r^{-1} e^{\mathsf{i} t}
\end{array}
\right)
\ \vert \ r, t \in \BR
\right\}
\end{equation}
and the unipotent radical $N$ of $B$
is given by 
\begin{equation} \label{unipotent}
N:=\left\{
\left(
\begin{array}{ccc}
1 & \overline{\nu} & \xi \\
 & 1 & \nu \\
 & & 1
\end{array}
\right)
\ \vert \ \nu, \xi \in \BC
\right\}
\end{equation}

The maximal torus $T$ decomposes as the product of the maximal split torus
\begin{equation*} 
A:=\left\{
\left(
\begin{array}{ccc}
r & & \\
 & 1 & \\
 & & r^{-1} 
\end{array}
\right)
\ \vert \ r \in \BR
\right\}
\end{equation*}
and of the compact torus
\begin{equation} \label{compactorus}
M:=\left\{
\left(
\begin{array}{ccc}
e^{\mathsf{i} t} & & \\
 & e^{-2\mathsf{i} t} & \\
 & & e^{\mathsf{i} t} 
\end{array}
\right)
\ \vert \ t \in \BR
\right\}
\end{equation}

\begin{rmk}\label{comptorus}
The torus $M$ coincides with $T \cap K_{\infty}^{\gamma^{-1}}$. It is fixed under conjugation by $\gamma$ and hence coincides with $T^{\gamma} \cap K_{\infty}$ as well. 
\end{rmk}

\begin{rmk}\label{iwasawa}
The Iwasawa decomposition of $G$ provides diffeomorphisms
\[
G \simeq K_{\infty}^{\gamma^{-1}} AN, \ \ \ G \simeq K_{\infty}A^{\gamma}N^{\gamma}
\]
\end{rmk}

\subsection{Lie algebras} Let $\mathfrak{g}$ be the Lie algebra of $G$. Upon fixing a basis of $V$, it is given by
$\mathfrak{g}= \{ X \in \mathfrak{sl}_3(\BC) \vert \overline{X}^{\top} J + JX =0 \}$. Let $\mathfrak{l}$ denote the Lie algebra of $K_{\infty}$.

\begin{pro} \label{csl3}
Fix a diagonal basis of $V$.
\begin{enumerate}
\item A $\BR$-basis for the real Lie algebra $\mathfrak{l}$ is given by 
\begin{align}
& U_0=\frac{1}{2} 
\left(
\begin{array}{ccc}
\iu & 0 & 0 \\
0 & \iu & 0 \\
0 & 0 & -2\iu
\end{array}
\right),\quad 
 U_1=\frac{1}{2} 
\left(
\begin{array}{ccc}
0 & \iu & 0 \\
\iu & 0 & 0 \\
0 & 0 & 0
\end{array}
\right),\\ 
& U_2=\frac{1}{2} 
\left(
\begin{array}{ccc}
0 & 1 & 0 \\
-1 & 0 & 0 \\
0 & 0 & 0
\end{array}
\right),\quad 
 U_3=\frac{1}{2} 
\left(
\begin{array}{ccc}
\iu & 0 & 0 \\
0 & -\iu & 0 \\
0 & 0 & 0
\end{array}
\right) \,.\nonumber
\end{align}
Moreover, $U_0$ and $U_3$ generate the Lie subalgebra $\mathfrak{m}$ given by the Lie algebra of the compact torus $M$ (cfr. Remark \ref{comptorus}).
We will use 
\begin{equation} \label{gen_l}
U_0, \ \ U_1 + \iu U_2, \ \ U_1-\iu U_2, \ \ U_3
\end{equation}
as generators of $\mathfrak{l}_{\BC}$ over $\BC$. 
\item Define the Lie algebra $\mathfrak{p} := \mathfrak{g} / \mathfrak{l}$. The Cartan decomposition provides an isomorphism of $\BR$-vector spaces 
\begin{equation} \label{dec_g}
\mathfrak{g} \simeq \mathfrak{l} \oplus \mathfrak{p}
\end{equation}
under which a $\BR$-basis of $\mathfrak{p}$ is given by
\begin{align}
& Y_{1}=
\left(
\begin{array}{ccc}
0 & 0 & 1 \\
0 & 0 & 0 \\
1 & 0 & 0
\end{array}
\right),\quad 
Y_{2}= 
\left(
\begin{array}{ccc}
0 & 0 & \iu \\
0 & 0 & 0 \\
-\iu & 0 & 0
\end{array}
\right),\\
& Y_{3}= 
\left(
\begin{array}{ccc}
0 & 0 & 0 \\
0 & 0 & 1 \\
0 & 1 & 0
\end{array}
\right),
Y_{4}=
\left(
\begin{array}{ccc}
0 & 0 & 0 \\
0 & 0 & \iu \\
0 & -\iu & 0
\end{array}
\right)\,. \nonumber
\end{align}

\item There is an isomorphism
\begin{equation} \label{realform}
\mathfrak{sl}_{3,\BC} \simeq \mathfrak{g}_{\BC} \simeq \mathfrak{g} \oplus \mathsf{i} \mathfrak{g}
\end{equation}
Under the complexification of \eqref{dec_g}, the Lie algebra
$\mathfrak{p}_{\BC} \simeq (\mathfrak{g} / \mathfrak{l})_{\BC}$ is generated over $\BC$ by
\begin{align}
& X_{1}=
\left(
\begin{array}{ccc}
0 & 0 & 1 \\
0 & 0 & 0 \\
0 & 0 & 0
\end{array}
\right),\quad 
X_{2}= 
\left(
\begin{array}{ccc}
0 & 0 & 0 \\
0 & 0 & 1 \\
0 & 0 & 0
\end{array}
\right),\\
& X_{3}= 
\left(
\begin{array}{ccc}
0 & 0 & 0 \\
0 & 0 & 0 \\
1 & 0 & 0
\end{array}
\right),\quad 
X_{4}=
\left(
\begin{array}{ccc}
0 & 0 & 0 \\
0 & 0 & 0 \\
0 & 1 & 0
\end{array}
\right)\,, \nonumber
\end{align}
so that
\begin{align} \label{gen_p}
 &    X_1=\frac{1}{2}(Y_1-\iu Y_2),  & X_2=\frac{1}{2}(Y_3-\iu Y_4), \\
 &    X_3=\frac{1}{2}(Y_1+\iu Y_2),  & X_4=\frac{1}{2}(Y_3+\iu Y_4)\,.\nonumber
\end{align}
Hence the complex conjugation $c$ on $\mathfrak{sl}_{3,\BC}$ 
induced by the isomorphism \eqref{realform} exchanges $X_1$ with $X_3$ and $X_2$ with $X_4$. 
\item In the Lie algebra $\mathfrak{g}_{\BC}$, we have
\begin{equation} \label{zero_bracket}
[X_i,X_j]=0 \ \mbox{mod} \ \mathfrak{l}_{\BC} \ \forall i,j \in \{1, \dots, 4 \}\,.
\end{equation}

\item Under the complexification of \eqref{dec_g}, the action of $\mathfrak{l}_{\BC}$ on $\mathfrak{g}_{\BC}$ by Lie bracket  stabilizes the Lie algebra $\mathfrak{p}_{\BC}$. Its action on the generators $X_i$ is recorded in the following table:
{\begin{center}
\begin{table} [htbp] 
\caption{Action of $\mathfrak{l}_{\BC}$ on $\mathfrak{p}_{\BC}$}\label{l_on_p}
\renewcommand{\arraystretch}{2}
\begin{tabular}{|c|c|c|c|c|}
 \hline
 & $U_0$ & $U_1+\iu U_2$ & $U_1- \iu U_2$ & $U_3$ \\
\hline
$X_1$ & $\frac{3 \iu}{2}X_1$ & $0$ & $\iu X_2$ & $\frac{\iu}{2}X_1$ \\
$X_2$ & $\frac{3 \iu}{2}X_2$ & $\iu X_1$ & $0$ & $-\frac{\iu}{2}X_2$ \\
$X_3$ & $-\frac{3 \iu}{2}X_3$ & $-\iu X_4$ & $0$ & $-\frac{\iu}{2}X_3$ \\
$X_4$ & $-\frac{3 \iu}{2}X_4$ & $0$ & $-\iu X_3$ & $\frac{\iu }{2}X_4$ \\
\hline
\end{tabular}
\end{table}
\end{center}
}
We then get a decomposition into a direct sum of commutative subalgebras, exchanged the one with the other by the complex conjugation $c$,
\begin{equation} \label{+-}
\mathfrak{p}_{\BC} \simeq \mathfrak{p}^+ \oplus \mathfrak{p}^-
\end{equation}
where $\mathfrak{p}^+$ is the subalgebra of $\mathfrak{p}_{\BC}$ generated by $X_1,X_2$ and $\mathfrak{p}^-$ is the subalgebra of $\mathfrak{p}_{\BC}$ generated by $X_3,X_4$.
\end{enumerate}
\end{pro}
\section{Principal series representations of SU$(2,1)$}\label{princ}
We now recall the structure of the principal series representations of $G$. Let $\phi_{\infty}$ be a character of $T$. We extend $\phi_{\infty}$ to the standard Borel $B$ of $G$ by letting it be trivial on the unipotent radical $N$. Then we consider the space\footnote{Note that we always work with \emph{algebraic}, i.e. non-unitarily normalized, induction.}\label{induction}
\begin{equation} \label{princseries}
I_{\phi_{\infty}} := \{ f : G \rightarrow \BC \ \mbox{smooth} | \ f(gb) = \phi^{-1}_{\infty}(b) f(g) \ \forall \ b \in B, g \in G\}
\end{equation}

\begin{defi}
The \emph{principal series representation} of $G$ associated to $\phi_{\infty}$ is the space $I_{\phi_{\infty}}$ endowed with the action of $G$ by \emph{left translations} $g.f : x \mapsto f(g^{-1}x)$.
\end{defi}

\begin{rmk}
The above representation is isomorphic as a $G$-representation to the space $\left\{ f : G \rightarrow \BC \ \mbox{smooth} | \ f(bg) = \phi_{\infty}(b) f(g) \ \forall \ b \in B, g \in G\right\}$
endowed with the action of $G$ by \emph{right translations} $g.f : x \mapsto f(xg)$. We have chosen to describe it using the action by left translations so that our conventions are compatible with~\cite{Zha19}.
\end{rmk}

Our goal is to describe the subspace of $K_{\infty}$-finite vectors in $I_{\phi_{\infty}}$. To this end, we fix half-integers $j,n,m_1,m_2$ satisfying the following properties:
\[
j \in \frac{1}{2}\BZ_{\geq0}; n \in \frac{1}{2}\BZ \ \ \ \mbox{such that} \ j+n \in \BZ; \quad m_1, m_2 \in \{-j, -j+1, \dots, j \}.
\] 

Consider the $\BC$-valued smooth function $W^{j,n}_{m_1,m_2}$ on $K_{\infty} \simeq \mathrm{U}(2)$, parametrized as in~\eqref{param_K}, defined by 
\[
W^{j,n}_{m_1,m_2} : \left(\begin{smallmatrix}
	e^{\frac{\mathsf{i}}{2}(-\zeta-\phi-\psi)}\cos\frac{\theta}{2}&-e^{\frac{\mathsf{i}}{2}(-\zeta+\phi-\psi)}\sin\frac{\theta}{2}\\
	e^{\frac{\mathsf{i}}{2}(-\zeta-\phi+\psi)}\sin\frac{\theta}{2}&e^{\frac{\mathsf{i}}{2}(-\zeta+\phi+\psi)}\cos\frac{\theta}{2}
	\end{smallmatrix}\right)
 \mapsto c^j_{m_1}c^j_{m_2}e^{\mathsf{i} n\zeta}e^{\mathsf{i}(m_1\psi+m_2\phi)}d^{(j,n)}_{m_1,m_2}(\theta) 
\]
where $c^j_m = \sqrt{(j+m)!(j-m)!}$ and the function $d^{(j,n)}_{m_1,m_2}(\theta)$ is given by
\[
d^{(j,n)}_{m_1,m_2}(\theta)=\frac{\left(\sin\frac{\theta}{2}\right)^{m_1-m_2}\left(\cos\frac{\theta}{2}\right)^{m_1+m_2}}{(j+m_2)!(j-m_2)!}P^{(m_1-m_2,m_1+m_2)}_{j-m_1}(\cos\theta)
\] 
with $P^{\alpha,\beta}_c(x)$ denoting the \emph{Jacobi polynomial} of parameters $(\alpha, \beta, c)$ (cfr.~\cite[(3.65)-(3.72)]{BL81}). 

It can be proven (see for example \emph{op. cit.}) that the functions $W^{j,n}_{m_1,m_2}$ (called \emph{Wigner D-functions}) give the matrix coefficients for the irreducible finite-dimensional representations of $\mathrm{U}(2)$. As a consequence of the Peter-Weyl theorem, this implies that the family of Wigner $D$-functions, for varying $j,n,m_1,m_2$, provides a Hilbert space basis for $L^2(K_{\infty})$.

Now consider the Iwasawa decomposition $G \simeq K_{\infty}A^{\gamma}N^{\gamma}$ of Remark~\ref{iwasawa}
and use it to define the unique extension of the function $W^{j,n}_{m_1,m_2}$ to the whole of $G$, such that the resulting function belongs to $I_{\phi_{\infty}}$. We will still denote these functions by $W^{j,n}_{m_1,m_2}$.

From now on, we will fix an integer $k \geq 0$ and work with the character $\phi_{\infty}$ defined by
\begin{equation} \label{character_k}
\left(
\begin{array}{ccc}
re^{\iu t} & & \\
 & e^{-2 \iu t} & \\
 & & r^{-1} e^{\iu t}
\end{array}
\right) \mapsto r^{3} (e^{\iu t })^{(2k+3)} 
\end{equation}
Then, a computation shows the following (cfr.~\cite[4.1]{Zha19}, keeping into account that by definition (\cite[last display of page 17]{Zha19}), our $I_{\phi_{\infty}}$ corresponds to Zhang's $I(\chi_{\delta,\lambda})$ with $(\delta, \lambda) =(2k+3,1)$):
\begin{lemma} \label{Kfinite}
The $K_{\infty}$-finite vectors in $I_{\phi_{\infty}}$ are finite linear combinations of Wigner $D$-functions $W^{j,n}_{m_1,m_2}$ whose parameters $j,n,m_1,m_2$ with $j \in \frac{1}{2}\BZ_{\geq0}, n, m_1, m_2 \in \frac{1}{2}\BZ$
satisfy 
\begin{equation}\label{para1}
-3j-2k-3 \leq n \leq 3j-2k-3, m_1, m_2 \in \{-j, \dots, j \}, 
\end{equation}
and 
\begin{equation} \label{para2}
3m_2-2k-3=n
\end{equation}
\end{lemma}

In what follows, we will abusively identify $I_{\phi_{\infty}}$ and its subspace of $K_{\infty}$-finite vectors. The latter is equipped with an action of the Lie algebra $\mathfrak{g}$ by left differentiation and with a compatible action of $K_{\infty}$ by left translations, giving it the structure of a $(\mathfrak{g}, K_{\infty})$-module. 
\begin{notation}
We will write $\texttt{dl}(x)(f)$ for the action of $x \in \mathfrak{g}$ by left differentiation on $f$, 
\[
\texttt{dl}(x)(f)(g):= \frac{\dm}{\mbox{d} t}f(e^{-tx}g)_{\vert_{t=0}} 
\]
\end{notation}
First, let us record the formulae describing the induced action of the Lie algebra $\mathfrak{l}_{\BC}$. 

\begin{lemma} \label{l-act}{(~\cite[Eq.~(24)-(25)]{Zha19})}
The action of the generators of $\mathfrak{l}_{\BC}$ (defined in \eqref{gen_l}) by left differentiation on $W_{m_1,m_2}^{j,n}$ is given by
\begin{align}
& \texttt{dl}(U_0)(W_{m_1,m_2}^{j,n})=\iu n W_{m_1,m_2}^{j,n},  \nonumber\\
& \texttt{dl}(U_3)(W_{m_1,m_2}^{j,n})= \iu m_1 W_{m_1,m_2}^{j,n}\\ \nonumber
& \texttt{dl}(U_1 \pm \iu U_2)(W_{m_1,m_2}^{j,n} )= -\iu \sqrt{(j \mp m_1)(j \pm m_1 +1)}\, W_{m_1 \pm 1,m_2}^{j,n} \,. \nonumber
\end{align}
\end{lemma}

Second, by translating its statement into our notations, we recall the main result of~\cite{Zha19}. It provides an explicit description of the action of the generators of $\mathfrak{p}_{\BC}$ (as defined in~\eqref{gen_p}) on $I_{\phi_{\infty}}$. 

\begin{rmk}  
To make the notational comparison with~\cite[Prop.~4.1]{Zha19}, we proceed as follows : 
\begin{enumerate}
\item we use Table~\ref{notation},
\begin{table}[htbp]
\caption{Notational Comparisons}\label{notation}
\renewcommand{\arraystretch}{2}
\begin{tabular}{ | c | c | c | c | } 
\hline
 Our notation & Notation of~\cite{Zha19} & Sign & $m_{\alpha}$ \\
 \hline
  $X_1$ & $v_{\alpha_1 + \alpha_2}$ & $+$ & $\frac{1}{2}$  \\
\hline
$X_2$ & $-v_{\alpha_2}$ & $+$ & $-\frac{1}{2}$ \\
 \hline
$X_3$ & $v_{-\alpha_1-\alpha_2}$ & $-$ & $-\frac{1}{2}$  \\
 \hline
 $X_4$ & $v_{-\alpha_2}$ & $-$ & $\frac{1}{2}$ \\
\hline
\end{tabular}
\end{table}
\item we remember that by definition, given our choice of $I_{\phi_{\infty}}$, the integer $\lambda$ which appears in~\cite{Zha19} is equal to 1 (compare (\cite[last display of page 17]{Zha19})),
\item we observe that the correct formula is not the one appearing in the \emph{statement} of~\cite{Zha19}, but the one appearing at the end of its \emph{proof} (\cite[page~21]{Zha19}). 
\end{enumerate}
\end{rmk}

\begin{thm}{\cite[Prop.~4.1]{Zha19}}\label{g-action}
Let $W^{j,n}_{m_1,m_2}$ be a Wigner $D$-function in $I_{\phi_{\infty}}$. Then the action of the generators of $\mathfrak{p}_{\BC}$ by left differentiation on $W_{m_1,m_2}^{j,n}$ is given by
\begin{align*}
& \texttt{dl}(X_1)(W_{m_1,m_2}^{j,n})= & \frac{1}{2(2j+1)}[-\sqrt{(j-m_1)(j-m_2)}(2j+m_2-n-1)W^{j-\frac{1}{2},n+\frac{3}{2}}_{m_1+\frac{1}{2},m_2+\frac{1}{2}} + \nonumber \\
& & + \sqrt{(j+m_1+1)(j+m_2+1)}(2j-m_2+n+3)W^{j+\frac{1}{2},n+\frac{3}{2}}_{m_1+\frac{1}{2},m_2+\frac{1}{2}} ] \nonumber \\
& \texttt{dl}(X_2)(W_{m_1,m_2}^{j,n})= & -\frac{1}{2 (2j+1)}[\sqrt{(j+m_1)(j-m_2)}(2j+m_2-n-1)W^{j-\frac{1}{2},n+\frac{3}{2}}_{m_1-\frac{1}{2},m_2+\frac{1}{2}} + \nonumber \\
& & + \sqrt{(j-m_1+1)(j+m_2+1)}(2j-m_2+n+3)W^{j+\frac{1}{2},n+\frac{3}{2}}_{m_1-\frac{1}{2},m_2+\frac{1}{2}}] \nonumber \\
& \texttt{dl}(X_3)(W_{m_1,m_2}^{j,n})= & \frac{1}{2 (2j+1)}[-\sqrt{(j+m_1)(j+m_2)}(2j-m_2+n-1)W^{j-\frac{1}{2},n-\frac{3}{2}}_{m_1-\frac{1}{2},m_2-\frac{1}{2}} + \nonumber \\
& & + \sqrt{(j-m_1+2)(j-m_2+1)}(2j+m_2-n+3)W^{j+\frac{1}{2},n-\frac{3}{2}}_{m_1-\frac{1}{2},m_2-\frac{1}{2}} ] \nonumber \\
& \texttt{dl}(X_4)(W_{m_1,m_2}^{j,n})= & \frac{1}{2 (2j+1)}[\sqrt{(j-m_1)(j+m_2)}(2j-m_2+n-1)W^{j-\frac{1}{2},n-\frac{3}{2}}_{m_1+\frac{1}{2},m_2-\frac{1}{2}} + \nonumber \\
& & + \sqrt{(j+m_1+1)(j-m_2+1)}(2j+m_2-n+3)W^{j+\frac{1}{2},n-\frac{3}{2}}_{m_1+\frac{1}{2},m_2-\frac{1}{2}} ] \nonumber \\
\end{align*}
\end{thm}

\section{Relative Lie algebra cohomology}\label{relativeLie}

We fix an integer $k \geq 0$ and keep the notation of the preceding section. In particular, we have a character $\phi_{\infty}$ of the maximal torus $T$ defined as in \eqref{character_k}, and the corresponding principal series representation $I_{\phi_{\infty}}$ of $G$, whose structure as a $(\mathfrak{g}, K_{\infty})$-module has been described in Lemmas~\ref{Kfinite} and~\ref{l-act} and in Theorem \ref{g-action}. The aim of this section, which is the heart of the paper, is to compute explicit generators of the $(\mathfrak{g}, K_{\infty})$-cohomology spaces valued in $I_{\phi_{\infty}} \otimes V_k$, where $V_k$ is the $k$-th symmetric power of $V$. 

\subsection{Finite dimensional $(\mathfrak{g}, K_{\infty})$-modules}
Choose a basis for $V$ and hence an isomorphism $V \simeq \BC^{3}$. Let $V_k$ be the space of complex degree-$k$ homogeneous polynomials in three variables, on which $g \in G$ if represented by a matrix $M_g$ with respect to the given basis, acts via 
$$g.p(x,y,z):= p((M_g)^{\top}(x,y,z)^{\top}).$$ It is an irreducible representation of the real Lie group $G$, isomorphic to the representation $\Sym^k V$. By differentiating the above action, $V_k$ acquires the structure of a $(\mathfrak{g}, K_{\infty})$-module.

Let us describe (a part of) its structure explicitly, by choosing a diagonal basis for $V$ and by complexifying the Lie algebras coming into play. We will only make use of the action of the generators of $\mathfrak{l}_{\BC}$ and $\mathfrak{p}_{\BC}$ on monomials of the form $x^{k-l}y^l$, for $l \in \{0, \dots, k \}$. The action of the generators of $\mathfrak{l}_{\BC}$ (cfr.~\eqref{gen_l}) is then given by 

\begin{align} \label{act_poly}
U \cdot x^{k-l}y^l = \left\{ 
\begin{array}{lll}
\iu \frac{k}{2}x^{k-l} y^{l}  &U=U_0 \\
\iu l x^{k-l+1} y^{l-1} & U=U_1+\iu U_2 \\
\iu (k-l) x^{k-l-1} y^{l+1} & U=U_1-\iu U_2 \\
\iu \frac{k-2l}{2} x^{k-l} y^{l} & U=U_3
\end{array}
\right.
\end{align}

On the other hand, the action of the generators of $\mathfrak{p}_{\BC}$ (cfr.~\eqref{gen_p}) on $V_k$ is given by 

\begin{equation}\label{g_act_poly}
X \cdot x^{k-l}y^l = \left\{ 
\begin{array}{cc}
0 & X=X_1 \\ 
0 & X=X_2 \\ 
(k-l)x^{k-l-1}y^l z & X=X_3 \\ 
lx^{k-l}y^{l-1}z & X=X_4 \\ 
\end{array}
\right.
\end{equation}

\subsection{Relative Lie algebra cohomology and the Delorme isomorphism}

The tensor product of the two $(\mathfrak{g}, K_{\infty})$-modules $I_{\phi_{\infty}}$ and $V_k$ has a natural structure of $(\mathfrak{g}, K_{\infty})$-module, with action of $\mathfrak{g}$ on $f \otimes p(x,y,z) \in I_{\phi_{\infty}} \otimes V_k$ given by
\[
\ X \in \mathfrak{g}, \ X.(f \otimes p(x,y,z)) := \texttt{dl}(X)(f) \otimes p(x,y,z) + f \otimes \frac{\dm}{\mbox{d} t} p((e^{tX})^T(x,y,z)^T)_{\vert_{t=0}}
\]

\begin{defi} \label{Liecomplex}
Consider the action of $\mathfrak{l}$ on $\Lambda^{q+1}(\mathfrak{g} / \mathfrak{l})$ given, for all positive integers $q$, by
\[
U \in \mathfrak{l}, \ U.(X_0, \dots, X_q)=\sum_{i=0}^q (X_0, \dots, [U,X_i], \dots, X_q)
\]

The \emph{Chevalley-Eilenberg} complex computing \emph{relative Lie algebra cohomology} of the $(\mathfrak{g}, K_{\infty})$-module $I_{\phi_{\infty}} \otimes V_k$ is the complex
\[
\cdots \rightarrow \Hom_{\mathfrak{l}}(\Lambda^{q}(\mathfrak{g} / \mathfrak{l}), I_{\phi_{\infty}} \otimes V_k) \xrightarrow{\rm{d}_q} \Hom_{\mathfrak{l}}(\Lambda^{q+1}(\mathfrak{g} / \mathfrak{l}), I_{\phi_{\infty}} \otimes V_k) \rightarrow \cdots
\]
 where the differential sends $\psi \in \Hom_{\mathfrak{l}}(\Lambda^{q}(\mathfrak{g} / \mathfrak{l}), I_{\phi_{\infty}} \otimes V_k)$ to  
\begin{align} \label{diff}
\rm{d}_q(\psi)(X_0, \dots, X_q) &= \sum_{i=0}^{q} (-1)^i X_i.\psi(X_0,\dots,\hat{X}_i,\dots, X_q) +\\
& + \sum_{i<j} (-1)^{i+j} \psi([X_i,X_j],X_0, \dots,\hat{X_i},\dots,\hat{X}_j,\dots,X_q) \nonumber
\end{align}
with the hat denoting an omitted argument. 
\end{defi}

The classical approach to the computation of the cohomology of the above complex, denoted by $H^{\bullet}(\mathfrak{g}, K_{\infty}, I_{\phi} \otimes V_k)$ makes use of the following result, due to Delorme. 

\begin{thm}{\cite[III, Theorem 3.3]{BoWa}}\label{delorme}
Let $\mathfrak{t}$ be the Lie algebra of the maximal torus $T$ of $G$ (cfr.~\eqref{maxtorus}), $\mathfrak{m}$ the Lie algebra of its compact subtorus $M$ (cfr.~\eqref{compactorus}), and $\mathfrak{n}$ the Lie algebra of the unipotent radical $N$ (cfr.~\eqref{unipotent})
of the Borel $B$ of $G$. Then
\[
H^{\bullet}(\mathfrak{g}, K_{\infty}, I_{\phi_{\infty}} \otimes V_k) \simeq \Hom(\Lambda^{\bullet}(\mathfrak{t} / \mathfrak{m}), (H^{\bullet}(\mathfrak{n}, V_k) \otimes \BC_{\phi_{\infty}})(0))
\]
where $\BC_{\phi_{\infty}}$ is the one-dimensional $\BC$-representation of $T$ via the character $\phi_{\infty}$, and $(0)$ means taking the weight-zero space for the induced action of $\mathfrak{t}$.
\end{thm}
The Lie algebra cohomology $H^{\bullet}(\mathfrak{n}, V_k)$ and its structure as a $\mathfrak{t}$-representation can be computed thanks to a theorem of Kostant (see for example~\cite[Theorem 3.2.3]{Vog81}). One finds that 
\begin{equation*} 
(H^i(\mathfrak{n}, V_k) \otimes \BC_{\phi_{\infty}}) (0) \simeq \left\{ 
\begin{array}{ll}
\BC & i=2 \\ 
0 & \mbox{otherwise} \\ 
\end{array}
\right.
\end{equation*}
which yields the following corollary of Theorem \ref{delorme}. 
\begin{coro} \label{1dim}
We have 
\begin{equation*} 
H^i(\mathfrak{g}, K_{\infty}, I_{\phi_{\infty}} \otimes V_k) \simeq \left\{ 
\begin{array}{ll}
\BC & i=2 \\ 
0 & \mbox{otherwise} \\ 
\end{array}
\right.
\end{equation*}
\end{coro}

However, the above abstract isomorphism does not provide us with an explicit element of the degree-2 part of the Chevalley-Eilenberg complex of Definition~\ref{Liecomplex}, representing a generator of $H^2(\mathfrak{g}, K_{\infty}, I_{\phi_{\infty}} \otimes V_k)$. 

\subsection{Explicit representatives of a generator}
 We want to provide explicit representatives of a generator of $H^2(\mathfrak{g}, \mathfrak{l}, I_{\phi_{\infty}} \otimes V_k)$. In other words, we need to construct closed, non-exact elements of $\Hom_{\mathfrak{l}}(\Lambda^{2}(\mathfrak{g} / \mathfrak{l}), I_{\phi_{\infty}} \otimes V_k)$. 

We will proceed as follows (an instance of \emph{Weyl's unitary trick}). Using the isomorphism~\eqref{realform}, the action of $\mathfrak{g}$ on $I_{\phi_{\infty}} \otimes V_k$ can be $\BC$-linearly extended to an action of the Lie algebra $\mathfrak{sl}_{3,\BC}$ on that same space. With respect to the restriction of this action to $\mathfrak{l}_{\BC}$, any element of $\Hom_{\mathfrak{l}}(\Lambda^{2}(\mathfrak{g} / \mathfrak{l}), I_{\phi_{\infty}} \otimes V_k)$ can be $\BC$-linearly extended to an element of $\Hom_{\mathfrak{l}_{\BC}}(\Lambda^{2}(\mathfrak{g} / \mathfrak{l})_{\BC}, I_{\phi_{\infty}} \otimes V_k)$. Reciprocally, any 
 element of $\Hom_{\mathfrak{l}_{\BC}}(\Lambda^{2}(\mathfrak{g} / \mathfrak{l})_{\BC}, I_{\phi_{\infty}} \otimes V_k)$ uniquely determines, by restriction to $\mathfrak{g}$ along~\eqref{realform}, an element of $\Hom_{\mathfrak{l}}(\Lambda^{2}(\mathfrak{g} / \mathfrak{l}), I_{\phi_{\infty}} \otimes V_k)$. Hence, to accomplish our task we will construct appropriate  $\BC$-linear morphisms
$\Lambda^{2}(\mathfrak{g} / \mathfrak{l})_{\BC} \rightarrow I_{\phi_{\infty}} \otimes V_k$ 
 which are $\mathfrak{l}_{\BC}$-equivariant, and study their properties with respect to the $\BC$-linear extension of the differential of Definition~\ref{Liecomplex}.  

\begin{defi} Consider the generators of $\mathfrak{p}_{\BC}$ defined in~\eqref{gen_p}.  For $i<j \in \{1,2,3,4 \}$, we define $X_{ij}:=X_i \wedge X_j$. 
\end{defi}

\begin{pro}
By looking at Table~\ref{l_on_p}, we see that the  generators of $\mathfrak{l}_{\BC}$ act on the elements $X_{ij}$ of the $\BC$-basis of $\Lambda^2(\mathfrak{g} / \mathfrak{l})_{\BC}=\Lambda^2(\mathfrak{p}_{\BC})$ defined above according to  Table~\ref{l_on_p^2}.

The decomposition~\eqref{+-} induces a bigrading on $\Lambda^2(\mathfrak{p}_{\BC})$, defined by 
\begin{equation} \label{bigrading}
\Lambda^{p,q}(\mathfrak{p}_{\BC}):=\Lambda^p \mathfrak{p}^+ \oplus \Lambda^q \mathfrak{p}^-
\end{equation}
for $p+q=2$, which is preserved by the $\mathfrak{l}_{\BC}$-action.
\vspace{5pt}
{\begin{center}
\begin{table} [htbp]
\caption{Action of $\mathfrak{l}_{\BC}$ on $\Lambda^2(\mathfrak{p}_{\BC})$=$\Lambda^2(\mathfrak{g} / \mathfrak{l}) _{\BC}$}\label{l_on_p^2}
\renewcommand{\arraystretch}{2}
\begin{tabular}{ | c | c | c | c | c | }
 \hline
 & $U_0$ & $U_1+\iu U_2$ & $U_1-\iu U_2$ & $U_3$ \\
\hline
$X_{12}$ & $3 \iu X_{12}$ & $0$ & $0$ & $0$ \\
$X_{23}$ & $0$ & $\iu (X_{13}-X_{24})$ & $0$ & $-\iu X_{23}$ \\
$X_{34}$ & $-3\iu X_{34}$ & $0$ & $0$ & $0$\\
$X_{13}$ & $0$ & $-\iu X_{14}$ & $\iu X_{23}$ & $0$ \\
$X_{14}$ & $0$ & $0$ & $-\iu (X_{13}-X_{24})$ & $\iu X_{14}$ \\
$X_{24}$ & $0$ & $\iu X_{14}$ & $-\iu X_{23}$ & $0$\\
\hline
\end{tabular}
\end{table}
\end{center}
}
\end{pro}
\begin{notation}
To simplify notation, from now on in this subsection the action of an element $x \in \mathfrak{g}$ on a function $f \in I_{\phi_{\infty}}$ will be denoted $x \cdot f$.
\end{notation}

\begin{defi} For $l \in \{-1, 0, \dots, k+1 \}$, $W^k_{l} := W^{\frac{k}{2}+1,-\frac{k}{2}}_{-\frac{k}{2}+l, \frac{k}{2}+1}$.
\end{defi} \label{wig_k}
These functions belong indeed to $I_{\phi_{\infty}}$ since their parameters verify the conditions~\eqref{para1},~\eqref{para2}. We will use them to define a first candidate generator for our cohomology space. 

\begin{defi} $\,$
\begin{enumerate}
\item For $l \in \{0, \dots, k \}$, let us define $\alpha_{l}:=\frac{k-l+1}{k+1} \sqrt{l+1} \sqrt{\binom{k+1}{l}} .$
\item Define 
\begin{align*}
& w_{13}^k:= \sum\limits_{l=0}^{k} \alpha_l W^k_l \otimes x^{k-l}y^l , & w^k_{23}:=-\iu \cdot (U_1-\iu U_2) \cdot w^k_{13},\\
& w_{24}^k:=-w_{13}^k, & w^k_{14}:=\iu \cdot (U_1+ \iu U_2) \cdot w^k_{13}.
\end{align*}

\item Denote by $\psi^k \in \Hom(\Lambda^{2}(\mathfrak{g} / \mathfrak{l})_{\BC}, I_{\phi_{\infty}} \otimes V_k)$
the morphism defined on generators by
\begin{equation} \label{gen}
X_{12} \mapsto 0, X_{34} \mapsto 0, X_{ij} \mapsto w^k_{ij} \ \mbox{for} \ (i,j) \notin \{(1,2), (3,4) \}
\end{equation}
\end{enumerate}
\end{defi}

\begin{rmk} \label{unique}
At this point, we do not know yet whether $\psi^k$ is $\mathfrak{l}_{\BC}$-equivariant. Nevertheless, once we choose the image of $X_{13}$ under $\psi^k$, the choice of the images of $X_{24}, X_{23}, X_{14}$ is uniquely determined by the requirement that $\psi^k$ be $\mathfrak{l}_{\BC}$-equivariant.

\end{rmk}

It will be helpful to have a second candidate generator for our cohomology space. 
\begin{defi} Consider, for $l \in \{0, \dots, k \}$, $W^k_{0,l} := W^{\frac{k}{2},-\frac{k}{2}-3}_{-\frac{k}{2}+l, \frac{k}{2}}$.
\end{defi} \label{wig_0,k}
These functions belong indeed to $I_{\phi_{\infty}}$ since their parameters verify the conditions~\eqref{para1},~\eqref{para2}. 

\begin{defi}
For $l \in \{ 0, \dots, k \}$, take $\beta_l:=\sqrt{\binom{k}{l}}$. Then $$\psi^k_0 \in \Hom(\Lambda^2 (\mathfrak{g} / \mathfrak{l})_{\BC}, I_{\phi_{\infty}} \otimes V_k)$$ can be defined
as the morphism 
\[
X_{34} \mapsto w_{0,34}^k := \sum\limits_{l=0}^k \beta_l W^k_{0,l} \otimes x^{k-l}y^l\,, \quad 
X_{ij} \mapsto 0 \ \ \forall \ (i,j) \neq (3,4)\,.
\]
\end{defi}

We will prove at once the $\mathfrak{l}_{\BC}$-equivariance of $\psi^k$ and $\psi_0^k$, and the fact that they are closed and non-exact, by considering some specific elements in the degree-1 part of the complex computing Lie algebra cohomology. 

\begin{defi} 
For $l \in \{0, \dots, k+1 \}$,  $
W^k_{\chi,l} := W^{\frac{k}{2}+\frac{1}{2},-\frac{k}{2}-\frac{3}{2}}_{-\frac{k}{2}-\frac{1}{2}+l, \frac{k}{2}+\frac{1}{2}}$.
\end{defi}
These functions belong indeed to $I_{\phi_{\infty}}$ since their parameters verify the conditions \eqref{para1}, \eqref{para2}.

\begin{defi}
For $l \in \{0, \dots, k \}$, take $\gamma_l := \sqrt{\frac{k+1-l}{k+1}} \sqrt{\binom{k}{l}}$. Then define $\chi^k \in \Hom((\mathfrak{g} / \mathfrak{l})_{\BC}, I_{\phi_{\infty}} \otimes V_k)$ as the morphism $X_1 \mapsto 0, \ \ \ X_2 \mapsto 0$,
\[
X_3 \mapsto \chi^k_3:= \sum\limits_{l=0}^k \gamma_l W^k_{\chi,l} \otimes x^{k-l}y^l,\ \ \ X_4 \mapsto \chi^k_4:= \iu (U_1 + \iu U_2) \chi^k_3\,.
\]
\end{defi}

\begin{lemma} \label{chi_equiv}
The morphism $\chi^k$
is $\mathfrak{l}_{\BC}$-equivariant, hence defines an element 
\[
\chi^k \in \Hom_{\mathfrak{l}_{\BC}}((\mathfrak{g} / \mathfrak{l})_{\BC}, I_{\phi_{\infty}} \otimes V_k) 
\]
\end{lemma}

\begin{proof}
In view of Table \ref{l_on_p}, in order to show $\mathfrak{l}_{\BC}$-equivariance of $\chi^k$, first we have to check that $ U_0 \cdot \chi_3^k = -\frac{3 \iu}{2} \chi_3^k, \ \ \ U_3 \cdot \chi_3^k = -\frac{\iu}{2} \chi_3^k$.
But these equalities follow from the fact that, by Lemma~\ref{l-act} and by~\eqref{act_poly}, 
\begin{align*}
& U_0 \cdot W_{\chi,l}^k = -\iu \left(\frac{k}{2}+\frac{3}{2}\right) W_{\chi,l}^k, & U_0 \cdot x^{k-l}y^l = \iu \frac{k}{2} x^{k-l}y^l,\\
& U_3 \cdot W_{\chi,l}^k = \iu \left(-\frac{k}{2}-\frac{1}{2}+l\right) W_{\chi,l}^k, & U_3 \cdot x^{k-l}y^l = \iu \left(\frac{k}{2}-l\right) x^{k-l}y^l\,.
\end{align*}

To complete the proof, Table~\ref{l_on_p} tells us that we need to show that 
\begin{align*}
& (U_1-\iu U_2) \cdot \chi^k_3 =0\\
& (U_1 + \iu U_2) \cdot ((U_1 + \iu U_2) \cdot \chi^k_3) = 0\\
&(U_1 - \iu U_2) \cdot ((U_1 + \iu U_2) \cdot \chi^k_3) = -\iu \chi^k_3 
\end{align*}

To compute the effect of $U_1 - \iu U_2$, observe that by Lemma~\ref{l-act}, for all $l \in \{ 0, \dots, k \}$
\[
(U_1-\iu U_2) \cdot W^k_{\chi,l}= -\iu \sqrt{l} \sqrt{k+2-l} W^k_{\chi,l-1}
\]
and that by~\eqref{act_poly}
\begin{align*}
& (U_1-\iu U_2) \cdot x^{k-l}y^l=\iu(k-l)x^{k-l-1}y^{l+1} \ \forall l \in \{ 0, \dots, k-1 \}\\
& (U_1-\iu U_2) \cdot y^k=0
\end{align*}
so that 
\begin{align*}
(U_1-\iu U_2) \cdot \chi^k_3 &= (U_1-\iu U_2) \cdot \sum\limits_{l=0}^k \gamma_l  W^k_{\chi,l} \otimes x^{k-l}y^l \\
& = -\iu[\sum\limits_{l=1}^k \gamma_l \sqrt{l} \sqrt{k+2-l} W^k_{\chi,l-1} \otimes x^{k-l}y^l - \sum\limits_{l=0}^{k-1} \gamma_l (k-l) W^k_{\chi,l} \otimes x^{k-l-1}y^{l+1}]\\
& = -\iu[\sum\limits_{l=1}^k (\gamma_l \sqrt{l} \sqrt{k+2-l}-\gamma_{l-1}(k-l+1)) W^k_{\chi,l-1} \otimes x^{k-l}y^l]
\end{align*}

which is indeed equal to zero since by definition, for all $l \in \{ 1, \dots, k \}$, we have $\gamma_l = \frac{k-l+1}{\sqrt{l} \sqrt{k-l+2}} \gamma_{l-1}.$ Using the fact that by Lemma~\ref{l-act},
for all $l \in \{ 1, \dots, k \}$
\[
(U_1+\iu U_2) \cdot W^{k}_{\chi,l}= -\iu \sqrt{k+1-l} \sqrt{l+1}  W^k_{\chi,l+1}
\]
and that by~\eqref{act_poly}
\begin{align*}
& (U_1+\iu U_2) \cdot x^k=0\\
& (U_1+\iu U_2) \cdot x^{k-l}y^l=\iu l x^{k-l+1}y^{l-1} \ \forall l \in \{ 1, \dots, k \}
\end{align*}

we see that 
\[
\chi^k_4= \iu (U_1 + \iu U_2) \cdot \chi^k_3 =\frac{\sqrt{k+1}}{k+1} W^k_{\chi,1} \otimes x^k + \sum\limits_{l=1}^{k-1} \gamma_{l+1} \frac{l+1}{k-l} W^k_{\chi,l+1} \otimes x^{k-l}y^l + W^k_{\chi,k+1} \otimes y^k
\]
Hence, 
\[
(U_1 + \iu U_2) \cdot ((U_1 + \iu U_2) \cdot \chi^k_3)=\left(-\frac{\sqrt{k+1}}{k+1} \sqrt{2k} + \frac{2}{k-1}\gamma_2\right) W^k_{\chi,2} \otimes x^k + 
\]
\[
+ \sum\limits_{l=2}^{k-1} \left(\frac{l(l+1)}{k-l} \gamma_{l+1}-l\sqrt{\frac{l+1}{k+1-l}} \gamma_l\right) W^k_{\chi,l+1} \otimes x^{k-l+1}y^{l-1} + 
\]
\[
+ \left(-k\sqrt{k+1}\gamma_k+k\right)W^k_{\chi,k+1} \otimes xy^{k-1}
\]
which is equal to $0$ since by definition
\begin{align*}
 &\gamma_2=(k-1) \frac{\sqrt{k}}{\sqrt{2}\sqrt{k+1}}\\
&\gamma_{l+1} = \frac{k-l}{\sqrt{l+1} \sqrt{k-l+1}} \gamma_{l} \ \ \ \forall l \in \{ 2, \dots, k-1 \}\\
& \gamma_k=\frac{1}{\sqrt{k+1}}
\end{align*}

A similar computation shows that $
(U_1 - \iu U_2) \cdot ((U_1 + \iu U_2) \cdot \chi^k_3)=-\iu \chi^k_3$ as desired, thus concluding the proof.
\end{proof}

Now, we want to study the effect of the differential $\rm{d}_1$, which we will denote by $\rm{d}$, on $\chi^k$. Given the expression of $\chi^k$, we need a preliminary lemma. 

\begin{lemma} \label{lie_wigner}
Recall the Wigner $D$-functions $W^k_l$ and $W^k_{0,l}$ from Definitions \ref{wig_k} and \ref{wig_0,k}. For all $l \in \{0, \dots, k+1 \}$, define
\[
\widetilde{W}^k_{0,l} := W^{\frac{k}{2}+1,-\frac{k}{2}-3}_{-\frac{k}{2}+l, \frac{k}{2}}
\]
Then the action of the elements $X_1, X_3, X_4$ of the Lie algebra $(\mathfrak{g} / \mathfrak{l})_{\BC}$ on 
\[
W^k_{\chi,l} = W^{\frac{k}{2}+\frac{1}{2},-\frac{k}{2}-\frac{3}{2}}_{-\frac{k}{2}-\frac{1}{2}+l, \frac{k}{2}+\frac{1}{2}}
\]
(defined for $l \in \{0, \dots, k \}$) is given as follows: 
\begin{align*}
 & X_1 \cdot W^k_{\chi,l} = \sqrt{\frac{l+1}{k+2}} W^k_l \ \ \forall l \in \{0, \dots, k \}\\
& X_3 \cdot W^k_{\chi,l} = \frac{\sqrt{l}\sqrt{k+1}}{k+2} W^k_{0,l-1} + \frac{k+3}{k+2}\sqrt{k+2-l} \widetilde{W}^k_{0,l-1} \ \ \forall l \in \{1, \dots, k+1 \}\\
& X_4 \cdot W^k_{\chi,l} = -\frac{\sqrt{k+1-l}\sqrt{k+1}}{k+2} W^k_{0,l} + \frac{k+3}{k+2} \sqrt{l+1} \widetilde{W}^k_{0,l} \ \ \forall l \in \{0, \dots, k \}
\end{align*}
\end{lemma}

\begin{proof}
It follows by direct computation from Theorem~\ref{g-action}.
\end{proof}

We are now ready to formulate and prove our main theorem. 

\begin{thm} \label{mainthm}
Write $\rm{d}$ for both $\rm{d}_1$ and $\rm{d}_2$. Then, the following hold.
\begin{enumerate} [label=(\roman*)] 
\item We have $\rm{d} \chi^k = \frac{1}{\sqrt{k+2}} \psi^k + \psi^k_0$.
\item We have $\rm{d} \psi^k_0 = 0 = \rm{d} \psi^k$.
\item The elements $\psi^k$ and $\psi^k_0$ do not belong to the image of $\rm{d}$.
\item Each one of the elements $\psi^k$ and $\psi^k_0$ defines a generator of $H^2(\mathfrak{g}, \mathfrak{l}, I_{\phi_{\infty}} \otimes V_k)$. The element $\psi^k$ is supported on $\Lambda^{1,1}(\mathfrak{p}_{\BC})$ and the element $\psi_0^k$ is supported on $\Lambda^{0,2}(\mathfrak{p}_{\BC})$ (as defined in~\eqref{bigrading}). 

\end{enumerate}
\end{thm}
\begin{proof}$\,$
($i$) By definition, and by using~\eqref{zero_bracket},
the morphism $\rm{d} \chi^k$ is such that
\[
X_{12} \mapsto X_1 \cdot \chi^k(X_2) - X_2 \cdot \chi^k(X_1) = 0
\]
\[
X_{34} \mapsto X_3 \cdot \chi^k(X_4) - X_4 \cdot \chi^k(X_3) 
\]
\[
X_{13} \mapsto X_1 \cdot \chi^k(X_3) - X_3 \cdot \chi^k(X_1) =  X_1 \cdot \chi^k(X_3)
\]
Hence, looking at the definition of $\psi^k$ and $\psi^k_0$, to show the formula in the statement, we have in particular to check that 
\[
X_3 \cdot \chi^k(X_4) - X_4 \cdot \chi^k(X_3) = \psi^k_0 (X_{34}) = \sum\limits_{l=0}^k \sqrt{\binom{k}{l}} W^k_{0,l} \otimes x^{k-l}y^l
\]
and that 
\[
X_1 \cdot \chi^k(3) = w_{13}^k= \frac{1}{\sqrt{k+2}}\sum\limits_{l=0}^{k} \frac{k-l+1}{k+1} \sqrt{l+1} \sqrt{\binom{k+1}{l}}  W^k_l \otimes x^{k-l}y^l
\]
This will be in fact \emph{sufficient} to show the desired formula, because of Remark~\ref{unique}. 

To perform the computation, we will use Lemma~\ref{lie_wigner} and the formulae for the action of the $X_i$'s on $V_k$ given in~\eqref{g_act_poly}.

To compute the value of $\rm{d} \chi^k$ on $X_{34}$, we get on one hand
\begin{align*}
X_3 \cdot \chi^k(X_4) & = X_3 \cdot \left(\frac{\sqrt{k+1}}{k+1} W^k_{\chi,1} \otimes x^k + \sum\limits_{l=1}^{k-1} \gamma_{l+1} \frac{l+1}{k-l} W^k_{\chi,l+1} \otimes x^{k-l}y^l + W^k_{\chi,k+1} \otimes y^k\right)\\
& = \frac{1}{k+2} W^k_{0,0} \otimes x^k + \sum\limits_{l=1}^{k-1} \gamma_{l+1} \frac{l+1}{k-l} \frac{\sqrt{l+1}\sqrt{k+1}}{k+2} W^k_{0,l} \otimes x^{k-l}y^l + \frac{k+1}{k+2} W^k_{0,k} \otimes y^k\\
& + \frac{k+3}{k+2}\widetilde{W}^0_{0,0} \otimes x^k + \sum\limits_{l=1}^{k-1} \gamma_{l+1} \frac{l+1}{k-l} \frac{k+3}{k+2} \sqrt{k+1-l}\widetilde{W}^k_{0,l} \otimes x^{k-l}y^l + \frac{k+3}{k+2} \widetilde{W}^k_{0,k} \otimes y^k \\
& + \frac{k}{k+1} \sqrt{k+1} W^k_{\chi,1} \otimes x^{k-1} z + \sum\limits_{l=1}^{k-1} \gamma_{l+1} (l+1) W^k_{\chi,l+1} \otimes x^{k-l-1}y^l z
\end{align*}
and on the other hand
\begin{align*}
X_4 \cdot \chi^k(X_3)  = & X_4 \cdot \left(\sum\limits_{l=0}^k \gamma_l W^k_{\chi,l} \otimes x^{k-l}y^l\right)\\
 = & -\sum\limits_{l=0}^k \gamma_l \frac{\sqrt{k+1-l} \sqrt{k+1}}{k+2} W^k_{0,l} \otimes x^{k-l}y^l
+ \sum\limits_{l=0}^k \gamma_l \frac{k+3}{k+2} \sqrt{l+1} \widetilde{W}^k_{0,l} \otimes x^{k-l} y^l\\
&\, + \sum\limits_{l=1}^k l\gamma_l W^k_{\chi,l} \otimes x^{k-l}y^{l-1}z 
\end{align*}

so that after using $\gamma_0=1, \gamma_1 = \frac{k}{k+1} \sqrt{k+1}, \gamma_k = \frac{1}{\sqrt{k+1}}$ in the above two equations, we get
{\scriptsize
\begin{align*}
X_3 \cdot \chi^k(X_4) - X_4 \cdot \chi^k(X_3) = & W^k_{0,0} \otimes x^k + \sum\limits_{l=1}^{k-1} \frac{\sqrt{k+1}}{k+2}\left(\gamma_{l+1} \frac{l+1}{k-l}\sqrt{l+1}+\gamma_l \sqrt{k+1-l}\right) W^k_{0,l} \otimes x^{k-l}y^l \\
& + W^k_{0,k} \otimes y^k + \sum\limits_{l=1}^{k-1} \frac{k+3}{k+2}\left(\gamma_{l+1} \frac{l+1}{k-l}\sqrt{k+1-l}-\gamma_l \sqrt{l+1}\right) \widetilde{W}^k_{0,l} \otimes x^{k-l}y^l
\end{align*}}
Since $\gamma_{l+1} = \frac{k-l}{\sqrt{l+1} \sqrt{k-l+1}} \gamma_{l}\,, \forall l \in \{ 1, \dots, k-1 \}$,
the previous expression is simplified as follows
\begin{align*}
X_3 \cdot \chi^k(X_4) - X_4 \cdot \chi^k(X_3) = W^k_{0,0} \otimes x^k + \sum\limits_{l=1}^{k-1} {\sqrt\frac{k+1}{k+1-l}} \gamma_l W^k_{0,l} \otimes x^{k-l}y^l +W^k_{0,k} \otimes y^k
\end{align*}
which given the definition of $\gamma_l$ provides the desired equality
\begin{align*}
 X_3 \cdot \chi^k(X_4) - X_4 \cdot \chi^k(X_3) =  \sum\limits_{l=0}^{k} \sqrt{\binom{k}{l}} W^k_{0,l} \otimes x^{k-l}y^l \,. 
\end{align*}

Let us now compute the value of $\rm{d} \chi^k$ on $X_{13}$. We have
\begin{align*}
X_1 \cdot \chi^k(X_3) & =  X_1 \cdot \left(\sum\limits_{l=0}^k \gamma_l W^k_{\chi,l} \otimes x^{k-l}y^l\right) \\
& =\sum\limits_{l=0}^k \gamma_l \sqrt{\frac{l+1}{k+2}} W^k_l \otimes x^{k-l}y^l \\
& = \frac{1}{\sqrt{k+2}} \sum\limits_{l=0}^k  \sqrt{\frac{k-l+1}{k+1}} \sqrt{l+1} \sqrt{\binom{k}{l}} W^k_l \otimes x^{k-l}y^l \\
& = \frac{1}{\sqrt{k+2}} \sum\limits_{l=0}^k  \frac{k-l+1}{k+1} \sqrt{l+1} \sqrt{\binom{k+1}{l}} W^k_l \otimes x^{k-l}y^l
\end{align*}

as desired. 

($ii$) Let us look at the action of $X_1$ and $X_2$ on the functions $ W^k_{0,l} = W^{\frac{k}{2},-\frac{k}{2}-3}_{-\frac{k}{2}+l, \frac{k}{2}} $ for $l \in \{0, \dots, k \}$. By applying Theorem~\ref{g-action}, we get $X_1 \cdot W^k_{0,l} = 0 = X_2 \cdot W^k_{0,l}$
for all $l \in \{0, \dots, k \}$. This fact and an application of~\eqref{g_act_poly} then show that $\rm{d} \psi^k_0 = 0$ because the only equality to be verified is $ X_1 \cdot \psi^k_0 (X_{34}) = 0 = X_2 \cdot \psi^k_0 (X_{34})$. Then, using point $(i)$, we conclude that
\[
\rm{d} \psi^k = \sqrt{k+2}(-\rm{d} \psi_0 + \rm{d} \rm{d} \chi) = 0
\]

$(iii)$ By contradiction, suppose $\psi^k_0 = \rm{d} \chi^k_0$ for some $
\chi_0 \in \Hom_{\mathfrak{l}_{\BC}}\left((\mathfrak{g} / \mathfrak{l})_{\BC}, I_{\phi_{\infty}} \otimes V_k\right)$.
This means in particular that 
\[
X_3 \cdot \chi^k_0(X_4) - X_4 \cdot \chi^k_0(X_3) = \rm{d} \chi^k_0(X_{34}) = \psi^k_0 (X_{34}) = \sum\limits_{l=0}^k \sqrt{\binom{k}{l}} W^k_{0,l} \otimes x^{k-l}y^l
\]
Now, both $\chi^k_0(X_4)$ and $\chi^k_0(X_3)$ are linear combinations of tensor products of Wigner $D$-functions and degree-$k$ monomials in the variables $x, y, z$. Suppose that a Wigner $D$-function $W^{j,n}_{m_1,m_2}$ contributes to $\chi^k_0(X_4)$. Then, by Theorem \ref{g_act_poly}, the function $X_3 \cdot W^{j,n}_{m_1,m_2}$ must have the form
\[
\alpha W^{j-\frac{1}{2},n-\frac{3}{2}}_{m_1-\frac{1}{2},m_2-\frac{1}{2}} + \beta W^{j+\frac{1}{2},n-\frac{3}{2}}_{m_1-\frac{1}{2},m_2-\frac{1}{2}}
\]
for some $\alpha, \beta \in \BC$. Hence, for one of the two Wigner $D$-functions involved in the latter expression to coincide with one of the $W^k_{0,l}$'s, i.e. with 
\[
W^{\frac{k}{2},-\frac{k}{2}-3}_{-\frac{k}{2}+l, \frac{k}{2}}
\]
for some $l \in \{0, \dots, k \}$, it is in particular necessary 
that 
\[
j \in \left\{\frac{k}{2}-\frac{1}{2},\frac{k}{2}+\frac{1}{2} \right\}, \ \ \ n=-\frac{k}{2}-\frac{3}{2}
\]
But remembering \eqref{para1} and \eqref{para2}, we must have $3m_2-(2k+3)=-\frac{k}{2}-\frac{3}{2}$ and $3j-(2k+3) \geq -\frac{k}{2}-\frac{3}{2}$, and this implies $m_2=\frac{k}{2}+\frac{1}{2}$, $j=\frac{k}{2}+\frac{1}{2}$. A similar analysis applies to the Wigner $D$-functions contributing to $\chi^k_0(X_3)$. 

We conclude that at least one of the elements $\chi^k_0(X_3)$ or $\chi^k_0(X_4)$ must include in its expression one of the functions $W_{\chi,l}^k, \ \ \ l \in \{ 0, \dots, k+1 \}$. 

Now observe that each of these functions is an eigenvector for the action of $U_0$ and $U_3$, and that the same holds for each degree-$k$ monomial in $x,y,z$. Since $\chi_0$ is $\mathfrak{l}_{\BC}$-equivariant and both $X_3$ and $X_4$ are eigenvectors for $U_0$ and $U_3$, the compatibility of the eigenvalues forces $\chi^k_0(X_3)$, if non-zero, to be a linear combination of 
\[
W^k_{\chi,l} \otimes x^{k-l}y^l
\]
for $l \in \{ 0, \dots, k \}$
and $\chi_0(X_4)$, if non-zero, to be a linear combination of 
\[
W^k_{\chi,l+1} \otimes x^{k-l}y^l
\]
for $l \in \{ 0, \dots, k \}$. 
Moreover, $\mathfrak{l}_{\BC}$-equivariance forces the conditions
\[
(U_1 - \iu U_2) \cdot \chi^k_0(X_3)=0\, \quad \text{and}\quad (U_1 + \iu U_2) \cdot \chi^k_0(X_4)=0
\]
By reverse-engineering the computation performed during the proof of Lemma \ref{chi_equiv}, we see that unless $\chi^k_0(X_3)=0$, the only possibility for the first condition to hold is that there exexists non-zero $\gamma \in \BC$ such that the equality 
\[
\chi^k_0(X_3) = 
\gamma\left(\sum\limits_{l=0}^k \sqrt{\binom{k}{l}} W^k_{0,l} \otimes x^{k-l}y^l\right)
\]
be verified. 

Now, since we are supposing $\psi^k_0 = \rm{d} \chi^k_0$, at least one of $\chi^k_0(X_3), \chi^k_0(X_4)$  must be non-zero. But as $\mathfrak{l}_{\BC}$-equivariance forces 
\[
(U_1 + \iu U_2) \cdot \chi^k_0(X_3)= - \iu \chi^k_0(X_4), \ \ \ (U_1 - \iu U_2) \cdot \chi^k_0(X_4)= - \iu \chi^k_0(X_3) 
\]
we see that $\chi^k_0(X_4) \neq 0$ if and only if $\chi^k_0(X_3) \neq 0$. Hence the images of $X_3$, $X_4$ under $\chi^k_0$ coincide, up to the scalar $\gamma$, with their images under the previously defined $\chi^k$. But then, using the computation of point $(ii)$, 
\[
\rm{d} \chi^k_0 (X_{13}) = X_1 \cdot \chi^k_0(X_3) = \gamma( X_1 \cdot \chi^k(X_3)) \neq 0
\]
which is absurd, since we are supposing $\rm{d} \chi^k_0(X_{13}) = \psi^k_0 (X_{13}) = 0$. This proves that $\psi^k_0$ cannot belong to the image of $\rm{d}$. Then, because of the equation
\[
\frac{1}{\sqrt{k+2}} \psi^k + \psi^k_0 = \rm{d} \chi^k
\]
the element $\psi^k$ cannot belong to the image of $\rm{d}$ either. 

($iv$) Given point $(i)$, of the fact that $\rm{d}$ preserves $\mathfrak{l}_{\BC}$-equivariance, of the support properties of $\psi^k$ and $\psi^k_0$ and of the fact that the bigrading of $\Lambda^2(\mathfrak{p}_{\BC})$ is preserved by the $\mathfrak{l}_{\BC}$-action, we see that both $\psi^k$ and $\psi^k_0$ are $\mathfrak{l}_{\BC}$-equivariant. Now the claim follows from points $(ii)$ and $(iii)$ and from the 1-dimensionality of $H^2(\mathfrak{g}, \mathfrak{l}, I_{\phi_{\infty}} \otimes V_k)$ (Corollary~\ref{1dim}).
\end{proof}

\section{Eisenstein classes on Picard surfaces} \label{eisclasses} 
This section aims to explain how to use the explicit representatives of a generator of $H^2(\mathfrak{g}, \mathfrak{l}, I_{\phi_{\infty}} \otimes V_k)$, provided by Theorem~\ref{mainthm}, to construct interesting differential forms on Picard modular surfaces. In particular, by definition, these differential forms will represent \emph{Eisenstein cohomology classes} on such surfaces. A good reference for everything we will say about Picard surfaces is~\cite{Gor92}.

\subsection{Picard surfaces} In what follows, we will fix a quadratic imaginary field $F$, and suppose that the 3-dimensional $\BC$-vector space $V$ is the extension of scalars to $\BC$ of a 3-dimensional $F$-vector space $V_F$, and that the hermitian form $J$ is the extension of scalars to $\BC$ of a $F$-valued hermitian form $J_F$ on $V_F$.  

We can then consider the algebraic group $\rG$ over $\BQ$ defined, for every $\BQ$-algebra $R$, by
\[
\rG(R)= \{ g \in \SL_{F \otimes_{\BQ} R}(V_F \otimes_{\BQ} R) \vert J_F(g \cdot, g \cdot)=J_F(\cdot, \cdot) \}
\]
Its real points satisfy $\rG(\BR)=G$
and we can consider its rational points $\rG(\BQ)$ as a subgroup of $G$. This provides us with a notion of \emph{congruence subgroup} of $G$. 

The coset space $X:=G / K_{\infty}$ is a  manifold diffeomorphic to an open ball in $\BC^2$, hence equipped with a natural complex structure. For any torsion-free congruence subgroup $\Gamma$ of $G$, the quotient
$S_{\Gamma}:=\Gamma \backslash X$
is a quasi-projective, smooth complex algebraic surface, called a \emph{Picard surface} (attached to $F$) of \emph{level} $\Gamma$. It can be identified with the $\BC$-points of a connected component of a Shimura variety attached to $\rG$, and as such, it turns out to be defined over a finite abelian extension of $F$. 

\subsection{Cohomology of local systems on Picard surfaces and Lie algebra cohomology} From now on, we fix a torsion-free congruence subgroup $\Gamma$ of $G$. For any integer $k \geq 0$, the symmetric powers $V_k=\Sym^k V$ define natural local systems on $S_{\Gamma}$. One is interested in studying the $V_k$-valued singular cohomology of the Picard surface $S_\Gamma$, denoted by $ H^{\bullet}(S_{\Gamma}, V_k)$. By~\cite[VII, Corollary 2.7]{BoWa}, there exists a natural isomorphism
\begin{equation} \label{isoBW}
H^{\bullet}(S_{\Gamma}, V_k) \simeq H^{\bullet}(\mathfrak{g}, \mathfrak{l}, \mathcal{C}^{\infty}(\Gamma \backslash G) \otimes V_k)
\end{equation}
where $\mathcal{C}^{\infty}(\Gamma \backslash G)$ denotes $\BC$-valued smooth functions on $\Gamma \backslash G$. 

\begin{rmk} \label{diff_forms}
The above isomorphism is constructed by first using the de Rham theorem to compute $H^{\bullet}(S_{\Gamma}, V_k)$ through $V_k$-valued smooth differential forms on $S_{\Gamma}$, and then providing a natural isomorphism between the $V_k$-valued de Rham complex and the complex computing relative Lie algebra cohomology of $\mathcal{C}^{\infty}(\Gamma \backslash G) \otimes V_k$. Because of this, by abuse of language, we will call \emph{differential forms} on $S_{\Gamma}$ the elements of the latter complex. Such a form will be called of \emph{type} $(p,q)$ if its $\BC$-linear extension to $\Lambda^2(\mathfrak{g} / \mathfrak{l})_{\BC}=\Lambda^2(\mathfrak{p}_{\BC})$ is supported on the $(p,q)$-part of the bigrading of $\Lambda^2(\mathfrak{p}_{\BC})$ defined in~\eqref{bigrading}.
\end{rmk}

Our goal is to see how to use the representatives of a generator of $H^2(\mathfrak{g}, \mathfrak{l}, I_{\phi_{\infty}} \otimes V_k)$, provided by Theorem~\ref{mainthm}, to construct differential forms, in the sense of the above remark, representing \emph{Eisenstein} classes in $H^{\bullet}(S_{\Gamma}, V_k)$ (to be defined in the next subsection).

To do so, it is convenient to switch to the adelic setting. Consider a compact open subgroup $K$ of $\rG(\BA_f)$ such that 
\begin{equation} \label{compactopen}
\Gamma = \rG(\BQ) \cap K
\end{equation}
We have then (see for example \cite[Proposition 4.18]{Milne2005})
\begin{equation} \label{adelic}
S_{\Gamma} \simeq \rG(\BQ) \backslash (X \times \rG(\BA_f) / K)
\end{equation}
Consider the space $\mathcal{C}^{\infty}(\rG(\BQ) \backslash \rG(\BA))$ of $\BC$-valued functions which are locally constant on $\rG(\BA_f)$ and smooth on the component at infinity $G$, and write $\mathcal{C}^{\infty}(\rG(\BQ) \backslash \rG(\BA))^K$ for the $K$-invariants with respect to the $\rG(\BA_f)$-action by right translation. Then, the isomorphism \eqref{adelic} allows us to rewrite \eqref{isoBW} as
\begin{equation} \label{isoBWadelic}
H^{\bullet}(S_{\Gamma}, V_k) \simeq H^{\bullet}(\mathfrak{g}, \mathfrak{l}, \mathcal{C}^{\infty}(\rG(\BQ) \backslash \rG(\BA))^K \otimes V_k)
\end{equation}

\subsection{Eisenstein cohomology of Picard surfaces}
The Picard surface $S_{\Gamma}$ is not compact. It embeds as an open submanifold of a canonical, compact real manifold with corners $
\overline{S_{\Gamma}}$
called the Borel-Serre compactification of $S_{\Gamma}$, having the property that the open immersion $j : S_{\Gamma} \hookrightarrow \overline{S_{\Gamma}}$ is a homotopy equivalence~\cite{BoSe73}. Hence, the local system $V_k$ extends canonically to $\overline{S_{\Gamma}}$, and there is an isomorphism $H^{\bullet}(S_{\Gamma}, V_k) \simeq H^{\bullet}(\overline{S_{\Gamma}}, V_k)$. If $\partial \overline{S_{\Gamma}}$ denotes the boundary $\overline{S_{\Gamma}} \setminus S_{\Gamma}$, this yields a restriction map towards \emph{boundary cohomology}
\begin{equation}\label{restrmap}
r: H^{\bullet}(S_{\Gamma}, V_k) \rightarrow H^{\bullet}(\partial \overline{S_{\Gamma}}, V_k)
\end{equation}

\begin{defi}\label{eis_def}
The image of $r$ is called~\emph{Eisenstein cohomology} of $S_{\Gamma}$ (with values in $V_k$) and denoted by
$H^{\bullet}_{\Eis}(S_{\Gamma}, V_k)$.
Classes in $H^{\bullet}(S_{\Gamma}, V_k)$ not belonging to the kernel of $r$ are then called \emph{Eisenstein classes}. 
\end{defi}

We want to recall Harder's description of $H^{\bullet}(\partial \overline{S_{\Gamma}}, V_k)$ and $H^{\bullet}_{\Eis}(S_{\Gamma}, V_k)$ (\cite{Har87b}). For this, we need to set up some language. Recall the maximal torus $T$ of $G$ defined in~\eqref{maxtorus}; it is the group of $\BR$-points of a maximal torus $\rT$ of the algebraic group $\rG$, arising as the Levi component of a Borel $\rB$ of $\rG$, with unipotent radical $\rN$ (with $\BR$-points given by the groups $B$ and $N$ of~\eqref{borel},~\eqref{unipotent}). 

\begin{defi}
An \emph{algebraic Hecke character} of $\rT$ is a continuous homomorphism $\phi : \rT(\BQ) \backslash \rT(\BA) \rightarrow \BC^{\times}$ such that its component at infinity $\phi_{\infty}$ is given (for $z \in \BC^{\times}$) by a character of the form
\begin{equation*}
\left(
\begin{array}{ccc}
\overline{z} & & \\
 & z \overline{z}^{-1} & \\
 & & z^{-1}
\end{array}
\right)
\mapsto z^{-a} \overline{z}^{-b}
\end{equation*}
where $a,b$ are integers. The couple $(a,b)$ is then called the \emph{infinity type} of $\phi$. If $\mu$ denotes the inverse of the above character, we will also say that such a $\phi$ is \emph{of type} $\mu$. 
\end{defi}
Whenever $\phi$ is a continuous $\BC^{\times}$-valued homomorphism on $\rT(\BQ) \backslash \rT(\BA)$, we can extend it to a continuous $\BC^{\times}$-valued homomorphism on $\rB(\BQ) \backslash \rB(\BA)$ by letting it be trivial on $\rN(\BA)$. We denote by $\phi_f$ the restriction of such a homomorphism to $\rB(\BA_f)$. 

\begin{defi}
Let $\phi$ be an algebraic Hecke character of $\rT$. We put
\[
I_{\phi, \BC} := \{ \mbox{locally constant} \  f: \rG(\BA_f) \rightarrow \BC \ \mbox{s. t.} \ f(bg)=\phi_f(b)f(g) \ \forall \ b \in \rB(\BA_f), g \in \rG(\BA_f) \}
\]
and for a compact open subgroup $K$ of $\rG(\BA_f)$, we denote by $I_{\phi,\BC}^K$ the $K$-invariants under the $\rG(\BA_f)$-action by right translation. 
\end{defi}

\begin{notation}
Let $W$ be the Weyl group of $\SL_3$. It is isomorphic to $S_3$, the symmetric group on 3 elements, and acts on characters of $\rT$ because of the isomorphism $\rG_F \simeq \SL_{3, F}$. The action of $w \in W$ on a character $\lambda$ is denoted by $w \cdot \lambda$. If $\rho$ is the half-sum of the positive roots of $\rT$, we put for $w \in W$, and $\lambda$ a character of $\rT$
\[
w \star \lambda := w \cdot (\lambda + \rho) - \rho
\]
\end{notation}

Let $K$ be as in~\eqref{compactopen}. The first result that we want to recall is a description of boundary cohomology. 
\begin{thm}{\cite[Eq.~(2.12),~Thm.~1]{Har87b}}
Let $k \geq 0$ be an integer and let $\lambda$ be the character of $\rT$ which induces the highest weight of the $G$-representation $V_k$. Then there is an isomorphism 
\[
H^{\bullet}(\partial \overline{S_{\Gamma}}, V_k) \simeq \bigoplus\limits_{w \in W} \bigoplus\limits_{\mbox{\emph{\tiny{type}}}(\phi)= w \star \lambda} I_{\phi, \BC}^K
\]
The space $I^K_{\phi,\BC}$ lives in cohomological degree $i$ if and only if $\phi$ is of type $w \star \lambda$ with the \emph{length} of $w$ being equal to $i$. 
\end{thm}
By computing the infinity types corresponding to the length 2 elements in the Weyl group, one gets the following.
\begin{coro}
Boundary cohomology in degree 2 is given by
\[
H^2(\partial \overline{S_{\Gamma}}, V_k) \simeq \bigoplus\limits_{\infty-\mbox{\emph{\tiny{type}}}(\phi)= (k,-k-3)} I_{\phi, \BC}^K \ \oplus \bigoplus\limits_{\infty-\mbox{\emph{\tiny{type}}}(\phi)= (-k-3,0)} I_{\phi, \BC}^K
 \]
\end{coro}

\begin{notation}
We normalize the isomorphism $W \simeq S_3$ in the following way: in the Corollary above, the permutation $(1\ 2\ 3)$ acts on $\lambda$ so that the corresponding infinity type is $(k,-k-3)$, and the permutation $(1\ 3\ 2)$ acts on $\lambda$ so that the corresponding infinity type is $(-k-3,0)$. We then call 
\[
\bigoplus\limits_{\infty-\mbox{\emph{\tiny{type}}}(\phi)= (k,-k-3)} I_{\phi, \BC}^K, \ \ \bigoplus\limits_{\infty-\mbox{\emph{\tiny{type}}}(\phi)= (-k-3,0)} I_{\phi, \BC}^K
\]
respectively the $(1\ 2\ 3)$-\emph{part} and $(1\ 3\ 2)$-\emph{part} of boundary cohomology. 
\end{notation}

\begin{rmk} \label{char_inf}
If the infinity type of $\phi$ is $(k,-k-3)$, then by definition the component at infinity of $\phi$ coincides with the character $\phi_{\infty}$ of \eqref{character_k}. One has the principal series representation $I_{\phi_{\infty}}$~\eqref{princseries} and we know that the space $H^2(\mathfrak{g}, \mathfrak{l}, I_{\phi_{\infty}} \otimes V_k)$ is one-dimensional (Corollary \ref{1dim}). An argument along the same lines shows that when $\phi$ has infinity type $(-k-3,0)$, the corresponding space $H^2(\mathfrak{g}, \mathfrak{l}, I_{\phi_{\infty}} \otimes V_k)$ is one-dimensional as well.  
\end{rmk}

We can now recall a second and most crucial result, which is a part of Harder's general description of Eisenstein cohomology of Picard surfaces. In the statement, we make use of the natural embedding of the idèles $\BI:= \BA^{\times}$ of $\BQ$
\begin{align}\label{restrQ}
\BI & \hookrightarrow \rT(\BA) \quad \text{defined by} \quad
x  \mapsto \left(
\begin{array}{ccc}
x & & \\
 & 1 & \\
 & & x^{-1}
\end{array}
\right) 
\end{align}
and the norm on $\BI$ is denoted by $\vert \cdot \vert_{\BI}$. Moreover, the $L$-function $L(\phi,s)$ of an algebraic Hecke character $\phi$ makes here its appearance. 
\begin{thm}{\cite[Thm.~2 and its proof]{Har87b}} \label{thmhard2}
Let $\phi$ be an algebraic Hecke character of $\rT$. Suppose that 
\begin{itemize}
\item either the infinity type of $\phi$ is $(-k-3,0)$,
\item or the infinity type of $\phi$ is $(k,-k-3)$ and one of the two following conditions holds: 
\begin{enumerate}
\item the restriction $\phi_{\BQ}$ of $\phi$ to $\BI$ along~\eqref{restrQ} satisfies $\phi_{\BQ} \neq \epsilon_{F \vert \BQ} \cdot \vert \cdot \vert^3_{\BI}$
where $\epsilon_{F \vert \BQ}$ denotes the quadratic character corresponding to $F \vert \BQ$;
\item we have $\phi_{\BQ} = \epsilon_{F \vert \BQ} \cdot \vert \cdot \vert^3_{\BI} \ \mbox{and} \ L(\phi,-1)=0$.
\end{enumerate}
\end{itemize}
Then for each integer $q$, there exist maps compatible with the differential
\[
\Eis_{\phi} : \Hom_{\mathfrak{l}}(\Lambda^q(\mathfrak{g} / \mathfrak{l}), I^K_{\phi,\BC} \otimes I_{\phi_{\infty}} \otimes V_k) \rightarrow \Hom_{\mathfrak{l}}(\Lambda^q(\mathfrak{g} / \mathfrak{l}), \mathcal{C}^{\infty}(\rG(\BQ) \backslash \rG(\BA))^K \otimes V_k)
\]
inducing a map on cohomology
\[
\Eis_{\phi} : I^K_{\phi,\BC} \otimes H^2(\mathfrak{g}, \mathfrak{l}, I_{\phi_{\infty}} \otimes V_k) \rightarrow H^2(S_\Gamma, V_k)
\]
such that for any choice of a generator $\Psi$ of $H^2(\mathfrak{g}, \mathfrak{l}, I_{\phi_{\infty}} \otimes V_k)$ (cfr.~Remark \ref{char_inf}) and resulting isomorphism 
\begin{equation}\label{isopsi}
\iota_{\Psi} : I_{\phi,\BC}^K \simeq I_{\phi,\BC}^K \otimes H^2(\mathfrak{g}, \mathfrak{l}, I_{\phi_{\infty}} \otimes V_k) 
\end{equation}
the composition $r \circ \Eis_{\phi} \circ \iota_{\Psi}$ with the restriction map~\eqref{restrmap} equals the identity of $I_{\phi,\BC}^K$.
\end{thm}

This implies: 
\begin{coro} \label{eisclass}
Let $\phi$ satisfy one of the conditions of Theorem~\ref{thmhard2}. For any non-zero $f \in I^K_{\phi,\BC}, \ \Psi \in H^2(\mathfrak{g}, \mathfrak{l}, I_{\phi_{\infty}} \otimes V_k)$, the class 
$\Eis_{\phi}(f \otimes \Psi) \in H^2(S_{\Gamma}, V_k)$ is an \emph{Eisenstein class} (Definition~\ref{eis_def}) restricting to $f \in H^2_{\Eis}(S_{\Gamma},V_k) \hookrightarrow H^2(\partial \overline{S_{\Gamma}}, V_k)$ at the boundary.  
\end{coro}

Let us indicate briefly how the morphisms $\Eis_{\phi}$ are constructed. By definition, we have an embedding $I^K_{\phi,\BC} \otimes I_{\phi_{\infty}} \hookrightarrow \mathcal{C}^{\infty}(\rB(\BQ) \backslash \rG(\BA))^K$ from which one would like to get a map $I^K_{\phi,\BC} \otimes I_{\phi_{\infty}} \rightarrow \mathcal{C}^{\infty}(\rG(\BQ) \backslash \rG(\BA))^K$ by sending $\Phi \in I^K_{\phi,\BC} \otimes I_{\phi_{\infty}}$ (seen as a smooth function on $\rB(\BQ) \backslash \rG(\BA)$) to the function defined for $g \in \rG(\BA)$ by
\[
g \mapsto \sum\limits_{\gamma \in \rB(\BQ) \backslash \rG(\BQ)} \Phi(\gamma g)
\]
The problem is that such an infinite sum may be divergent. To solve it, one fixes an isomorphism 
\begin{equation} \label{isoideles}
\BI_F \simeq \rT(\BA)
\end{equation}
whose component at infinity is given by sending $z \in \BC$ to
\begin{equation*}
\left(
\begin{array}{ccc}
\overline{z} & & \\
 & z \overline{z}^{-1} & \\
 & & z^{-1}
\end{array}
\right)
\end{equation*}
and considers the character
$\vert \cdot \vert : \rT(\BA) \rightarrow \BC^{\times}$ 
given by the adelic norm under \eqref{isoideles}. One extends this to a map $\vert \cdot \vert : \rG(\BA) \rightarrow \BC^{\times}$, by considering a suitable maximal compact subgroup $K_{\BA}$ of $\rG(\BA)$ and the associated Iwasawa decomposition, and letting $\vert \cdot \vert$ be trivial on $\rN(\BA)$ and on $K_{\BA}$. Then it is known (\cite{Lan76}) that for $s \in \BC$ with large enough real part, the infinite sum\footnote{Note that as observed in the footnote in section \ref{induction}, we are not using unitary induction. This explains the absence of the half-sum of positive roots in our formulae.}
\[
\sum\limits_{\gamma \in \rB(\BQ) \backslash \rG(\BQ)} \Phi(\gamma g) \vert \gamma g \vert^s
\]
is absolutely convergent,  thus allowing one to define an element
\[
\Eis_{\phi,s}(\Phi) \in \mathcal{C}^{\infty}(\rG(\BQ) \backslash \rG(\BA))^K
\]
The content of Theorem~\ref{thmhard2} is to give conditions under which the limit
\[
\lim_{s \to 0} \Eis_{\phi,s}(\Phi)
\]
exists and induces, after tensoring with $V_k$, passing to $(\mathfrak{g}, \mathfrak{l})$-cohomology and using the isomorphism~\eqref{isoBWadelic}, a map $\Eis_{\phi}$ with the desired properties. The crucial tool to obtain such conditions is \emph{Langlands' constant term formula} (\cite{Lan76}), which shows how the $L$-function of $\phi$ controls convergence of the limit.

In view of the applications to the study of $L$-functions that we have in mind, the case of interest for us is the one of Eisenstein classes arising from a $\phi$ of infinity type $(k,-k-3)$. By putting together our main Theorem~\ref{mainthm} and Corollary~\ref{eisclass}, we get
\begin{coro} \label{finalcoro}
Fix an integer $k \geq 0$ and let $\phi$ be an algebraic Hecke character of $\rT$ of infinity type $(k,-k-3)$, satisfying one of the conditions of Theorem~\ref{thmhard2}. Let $\psi^k$ and $\psi_0^k$ be the elements considered in Theorem~\ref{mainthm}. For any non-zero $f \in I^K_{\phi,\BC}$, the differential forms on $S_{\Gamma}$ 
\[
\Eis_{\phi}(f \otimes \psi^k), \ \Eis_{\phi}(f \otimes \psi_0^k)
\]
provide representatives, of type respectively $(1,1)$ and $(0,2)$ in the sense of Remark~\ref{diff_forms}, of an Eisenstein class in $H^2(S_{\Gamma}, V_k)$, restricting to $f$ in $H^2(\partial \overline{S_{\Gamma}}, V_k)$.
\end{coro}

\section*{Acknowledgements}
JB would like to thank the Department of Mathematics, Kiel University, Germany for the hospitality and support. MC thanks the Foundation Mathématique Jacques Hadamard and the Laboratoire de Mathématiques d'Orsay for the excellent working conditions. Thanks are also due to G. Ancona, L. Clozel, J. Fresán, G. Harder, and V. Hernandez for discussions and comments on the first draft of this article.

\nocite{}
\bibliographystyle{abbrv}
\bibliography{BC}

\end{document}